\documentclass{amsart}
\usepackage{graphicx}
\usepackage{url}
\usepackage{amsthm}
\usepackage{amsmath}
\usepackage{amsfonts}
\usepackage{amssymb}
\usepackage{MnSymbol}
\usepackage{enumerate}
\usepackage{extarrows}
\usepackage{comment}
\usepackage{booktabs}
\usepackage{subcaption}
\usepackage{pgfplots}
\usepackage{standalone}
\pgfplotsset{compat=newest}

\usepackage[section]{placeins}

\definecolor{darkgreen}{rgb}{0,0.5,0}
\usepackage[colorlinks, citecolor=darkgreen, backref]{hyperref}
\theoremstyle{plain}
\newtheorem{theorem}{Theorem}
\numberwithin{theorem}{section}
\newtheorem{lemma}[theorem]{Lemma}

\newtheorem{conj}[theorem]{Conjecture}
\newtheorem*{conj*}{Conjecture}
\newtheorem{cor}[theorem]{Corollary}
\newtheorem{prop}[theorem]{Proposition}

\newtheorem*{lemma*}{Lemma}
\theoremstyle{definition}
\newtheorem{remark}[theorem]{Remark}

\newtheorem{conjx}{Conjecture}

\newcommand{\R}{\mathbb{R}}
\newcommand{\N}{\mathbb{N}}
\newcommand{\Z}{\mathbb{Z}}

\newcommand{\eps}{\varepsilon}
\newcommand{\floor}[1]{\left\lfloor #1 \right\rfloor}
\newcommand{\ceil}[1]{\left\lceil #1 \right\rceil}

\newcommand{\card}[1]{\# {#1}}

\begin{document}

\title[Partition functions and perfect powers: a heuristic approach]{Repellent properties of perfect powers on partition functions: a heuristic approach}

\begin{abstract}
In 2013, Sun conjectured that the partition function $p(n)$ is never a perfect power for $n \geq 2$. Building on this, Merca, Ono, and Tsai recently observed that for any fixed integers $d \geq 0$ and $k \geq 2$, there appear to be only finitely many integers $n$ such that $p(n)$ differs from a perfect $k$th power by at most $d$. Denoting by $M_k(d)$ the largest such $n$, they conjectured that $M_k(d) = o(d^\epsilon)$ for every $\epsilon > 0$.

In this paper, we investigate the asymptotic growth of analogs of $M_k(d)$ for a wide class of partition functions. We establish sharp lower bounds and provide heuristics which suggest that $M_k(d)$ in fact grows polylogarithmically in $d$, i.e. of order $\log^2(d)$. More generally, we prove that if $f(n)$ is a suitably random chosen function with asymptotic growth rate similar to that of $p(n)$, then the set of integers $n$ for which $f(n)$ is a perfect power is finite with probability 1. 
\end{abstract}

\author{Summer Haag}
\address{University of Colorado Boulder \\ Boulder \\ Colorado \\ USA}
\email{summer.haag@colorado.edu}

\author{Praneel Samanta}
\address{Department of Mathematics, University of Kentucky, Lexington, KY 40506, USA}
\email{praneel.s@uky.edu}

\author{Swati}
\address{Department of Mathematics, University of South Carolina, Columbia, SC 29208, USA}
\email{s10@email.sc.edu}

\author{Holly Swisher}
\address{Department of Mathematics, 368 Kidder Hall, Oregon State University, Corvallis, OR 97331, USA}
\email{swisherh@oregonstate.edu}

\author{Stephanie Treneer}
\address{Department of Mathematics, Western Washington University, Bellingham, WA 98225, USA}
\email{trenees@wwu.edu}

\author{Robin Visser}
\address{Charles University \\ Faculty of Mathematics and Physics \\ Department of Algebra \\ Sokolovsk\'{a} 83 \\ 186 75 Praha 8 \\ Czech Republic}
\email{robin.visser@matfyz.cuni.cz}

\date{\today}

\thanks{The authors were supported in part by NSF grants DMS-2418528 and DMS-2201085 through the 2025 Rethinking Number Theory Workshop.  The fourth author was also supported in part by NSF grant DMS-2101906. R. V. was also supported by {Charles University} programme PRIMUS/24/SCI/010 and {Charles University} programme UNCE/24/SCI/022.}



\maketitle

\markboth{Partition functions and perfect powers: a heuristic approach}{HAAG, SAMANTA, SWATI, SWISHER, TRENEER, VISSER}

\section{Introduction}

For each positive integer $n$, let $p(n)$ be the number of \emph{partitions} of $n$, that is the number of non-increasing sequences of positive integers $(\lambda_1, \lambda_2, \dots, \lambda_k)$ such that $\lambda_1 + \lambda_2 + \dots + \lambda_k = n$.  We further define $p(0)=1$.  In 1751, Euler obtained the following infinite product generating function for $p(n)$, 
\begin{equation}
    \sum_{n \geq 0} p(n) q^n = \prod_{n \geq 1} \frac{1}{1 - q^n},
\end{equation}
and in 1918, Hardy and Ramanujan \cite{HardyRamanujan} pioneered the circle method to prove the following asymptotic formula
\begin{equation}\label{eq:HRasymptotic}
    p(n) \sim \frac{1}{4 n \sqrt{3}} \exp{\Bigg(\pi \sqrt{\frac{2n}{3}} \Bigg)} \; \text{ as } n \to \infty.
\end{equation}
This was later improved by Rademacher \cite{Rademacher}, who obtained an exact infinite series that converges to $p(n)$.  

The partition function has long been an object of broad and sustained interest, with connections to many areas of mathematics including mathematical physics (see \cite{Andrews, AndrewsEriksson, Apostol} for example).  And it continues to inspire intriguing observations. In 2013, Zhi-Wei Sun \cite{OEIS_A41} made the following conjecture (see also \cite{Sun_MathOverflow}, \cite[Conjecture~8.9~(iii)]{Sun_book} and \cite{Amdeberhan, Freddie}).

\begin{conjx}[Sun] \label{conj:sun} The partition function $p(n)$ is never a perfect power of the form $a^b$ for integers $a, b > 1$.
\end{conjx}

This conjecture has been verified for all $n \leq 10^8$ by Alekseyev \cite{Sun_MathOverflow}, but still remains an open problem.  

It's natural to consider whether analogs of Sun's conjecture hold for various different partition functions $f(n)$.  One category of restricted partition functions can be described by $p_A(n)$, the number of partitions of $n$ with parts taken from the set $A\subset\N$. The generating function for $p(A)$ is
\begin{equation*}
    \sum_{n=0}^{\infty} p_A(n) x^n = \prod_{a \in A} \frac{1}{1 - x^a}.
\end{equation*}
Similar conjectures to Sun's conjecture regarding whether $p_A(N)$ can be a perfect square were formulated by Tengely and Ulas \cite[Section~5]{TengelyUlas}.  In the case where $A$ is a finite set consisting of the first $k$ positive integers $\{1, 2, \dots, k\}$, the analog of Sun's conjecture was recently resolved by Ono \cite{Ono2025}.
  
In recent work expanding upon Sun's conjecture for $p(n)$, Merca, Ono and Tsai \cite{MercaOnoTsai} consider the proximity of $p(n)$ to perfect powers more generally. For fixed $n\geq 1$ and $k>1$, let
\begin{equation*}
    \Delta_k(n) := \min \left\{ | p(n) - m^k | : m \in \Z \right\},
\end{equation*}
the distance from $p(n)$ to the closest $k$th power. Merca-Ono-Tsai make the following conjecture about $\Delta_k(n)$. 

\begin{conjx} \label{conj:MOT1}(\cite{MercaOnoTsai}, Conj. 1) For fixed integers $k>1$ and $d\geq 0$, there are at most finitely many $n$ for which $\Delta_k(n)\leq d$.\end{conjx}

Assuming Conjecture \ref{conj:MOT1}, there is a well-defined maximal integer $n$ for which $p(n)$ is within $d$ of a $k$th power. Let
\begin{equation}
    M_k(d) := \max \{ n \in \N : \Delta_k(n) \leq d \}.
\end{equation}
Then a proof of Sun's conjecture would be equivalent to showing that $M_k(0)=1$ for all $k>1$. Merca-Ono-Tsai give a table of conjectured values of $M_k(d)$, but proving even a single value would be an achievement, since none are currently known. In this paper we investigate the growth of $M_k(d)$ for fixed $k$ as well as for fixed $d$, motivated by the following conjectures of Merca-Ono-Tsai. 

If we fix $k>1$, Merca-Ono-Tsai predict that for $p(n)$, $\{M_k(d)\}_{d=1}^\infty$ grows more slowly than any power of $d$.

\begin{conjx} \label{conj:MOT2} (\cite{MercaOnoTsai}, Conj. 2) For each integer $k>1$ and every real $\eps>0$, we have that $M_k(d)=o(d^{\eps})$.\end{conjx}

Here, we take a more general approach to these questions by considering any arithmetic function $f(n)$ with $f(n)\to\infty$ as $n\to\infty$ and defining analogous functions

\begin{align}
\Delta_{f,k}(n) &:= \min \left\{ | f(n) - m^k | : m \in \Z \right\}, \label{Deltafk} \\
M_{f,k}(d) &:= \max \{ n \in \N : \Delta_{f,k}(n) \leq d \}.  \label{Mfk}
\end{align}

In Proposition \ref{prop:Mkdlower} we prove a lower bound for $M_{f,k}(d)$ for fixed $k$ and sufficiently large $d$. In Proposition \ref{asy_Mk}, we are able to make this bound explicit when the asymptotic growth of $f(n)$ is comparable to that of $p(n)$. In particular, for $p(n)$, we obtain the following lower bound.

\begin{theorem} \label{thm:Mkd_intro}
    Let $k \geq 2$. We have for sufficiently large $d$,
    \begin{equation*}
        M_{k}(d) \geq \frac{3}{2 \pi^2}  \Big( \frac{k}{k-1} \Big)^2 (\log{d})^2 + O ( \log d \log \log d).
    \end{equation*}
\end{theorem}

Motivated by computational data discussed in Section \ref{sec:computations}, we make the following conjecture.  

\begin{conjx} \label{conj:Mkdasymp}
    Let $k \geq 2$.  Then $M_k(d) \sim \frac{3}{2 \pi^2} \Big( \frac{k}{k-1} \Big)^2 (\log{d})^2$ as $d \to \infty$.
\end{conjx}

When we instead fix $d$, Merca-Ono-Tsai assert that $M_k(d)$ stabilizes in a predictable way.  To state their conjecture, we first make the following definition. For each fixed integer $d\geq 0$, define $L(d)$ by
\begin{equation}\label{eq:Lp}
    L(d):=\max\{n : p(n)\leq d+1\}.
\end{equation} 

\begin{conjx} \label{conj:MOT3&4} (\cite{MercaOnoTsai}, Conj. 3 and 4) 
There is a positive integer $N_d$ such that for all $k\geq N_d$,
\[
M_k(d)=L(d). 
\]
\end{conjx}
\noindent Additionally, they observe that the bound $N_d$ appears to grow like $O(\log d)$. 

Again we take a more general approach, and in Proposition \ref{prop:limit} we prove a generalization of Conjecture \ref{conj:MOT3&4} in the cases when $M_{f,k}(d)$ is bounded as $k\to\infty$. Using this assumption we also give, in Theorem \ref{thm:ndlower}, an explicit lower bound for the smallest value of $k$ for which $\{M_{f,k}(d)\}_{k=1}^\infty$ stabilizes.  In particular, for the partition function $p(n)$, we prove the following.

\begin{theorem}\label{thm:Nd_intro}
    If $\{M_k(d)\}_{k=1}^\infty$ is bounded then there exists a smallest positive integer $N_d$ such that $M_k(d)=L(d)$ for all $k\geq N_d$. Furthermore, we have
    \[
    N_{d} \geq \lfloor \log_{2}(d) \rfloor + A \log \log d
    \]
    for any $0 < A < \frac{1}{\log 2}$ and for all sufficiently large $d$.
\end{theorem}

\subsection{A heuristic approach}
Proving an explicit upper bound for $\{M_k(d)\}_{d=1}^\infty$ remains elusive. More generally, proving precise values of $M_{f,k}(d)$ for generic $k$ and $d$, even for many simple functions such as $f(n) = \lfloor e^n \rfloor$,  is a very difficult problem, but heuristic methods can provide a more tractable approach. We can use a random variable to model the behavior of functions $f(n)$ similar to the partition function, then apply this heuristic to investigate which statements hold with probability 1. 

Such methods have been frequently used to model the prime numbers. Cram\'{e}r \cite{Cramer} originally considered a heuristic model of the prime numbers to give a conjectured growth rate for the maximal gap between consecutive prime numbers. Such models were later refined by Granville \cite{Granville} and more recently by Banks--Ford--Tao \cite{BanksFordTao}.

Heuristic models have also been used to study the statistics of zeros of $L$-functions (e.g. see the Katz-Sarnak philosophy \cite{KatzSarnak_book, KatzSarnak}), the asymptotic behaviour of the M\"{o}bius function $\mu(n)$ (e.g. see \cite{Sarnak}) as well as give predictions for class groups of number fields (e.g. see \cite{CohenLenstra} or \cite[p.~295]{Cohen_book}).\footnote{For a further explanation for such probabilistic models, we refer the reader to the excellent blog post of Tao \cite{Tao_blog}.}

We take care to emphasize that the results of this heuristic approach by no means constitute rigourous proofs regarding how close $f(n)$ is to perfect powers, but they can at least tell us \emph{morally speaking} which statements and conjectures we expect to be true, assuming that $f(n)$ behaves in a ``sufficiently random'' way, which we shall make precise below. 

In Section~\ref{sec:model}, we develop a general heuristic model for any sufficiently fast-growing function $f(n)$. Here we briefly describe that model for $p(n)$. We set $A(n) = \frac{1}{4 \sqrt{3} n} \exp(\pi \sqrt{\frac{2 n}{3}} )  $ , the main term of the Hardy-Ramanujan asymptotic for $p(n)$. Let $\{\eps_n\}_{n=1}^\infty$ be a sequence of positive real numbers such that $\eps_n \to 0$ as $n \to \infty$.  To control the rate at which $\eps_n$ converges to 0, we also assume that 
\begin{equation}\label{eq:intro_eps_ncondition}
\eps_n \gg \frac{\log A(n)}{\sqrt{A(n)}} \asymp n \exp{ \Big( - \frac{\pi}{2} \sqrt{ \frac{2n}{3}} \Big)} 
\end{equation}
as $n \to \infty$, where $\asymp$ denotes asymptotically equivalent up to a constant. Next we let $\{\mathbf{p}_n\}_{n=1}^\infty$ be a sequence of discrete random variables defined so that for each $n$, $\mathbf{p}_n$ takes a uniform distribution among the set 
\begin{equation}\label{eq:intro_S_ndef}
\mathcal{S}_n = \{x\in \N \mid A(n)(1 - \eps_n) \leq x \leq A(n) (1 + \eps_n) \}. 
\end{equation}

\noindent For any integer $m \in \Z$, we define the probability mass function
\begin{equation*}
    \mathbb{P}( \mathbf{p}_n = m) = \begin{cases}
        \frac{1 }{ \card{\mathcal{S}_n}} &\text{if } m \in \mathcal{S}_n, \\
        0 & \text{otherwise. }
    \end{cases}
\end{equation*}

\noindent We can now define the random variables
$$\mathbf{\Delta}_k(n) := \min \big\{  | \mathbf{p}_n- m^k| : m \in \Z  \big\} $$
and 
$$ \mathbf{M}_k(d) := \max \{ n :  \mathbf{\Delta}_k(n) \leq d \} $$
analogously to $\Delta_k(n)$ and $M_k(n)$.

We are then able to prove the following theorem for $p(n)$. 

\begin{theorem} \label{thm:introheuristic}
The following statements are each true with probability 1.
\begin{enumerate}[(a)]
\item For any $d\geq 0$ and $k\geq 2$, the random variable $\mathbf{M}_k(d)$ is finite.
\medskip
\item The sequence of random variables $\{\mathbf{p}_n\}_{n=1}^\infty$ will take only finitely many perfect powers.
\medskip
\item Let $k \geq 2$ and assume moreover than $\eps_n \gg A(n)^{-1/k}$ as $n \to \infty$.  For any $\eps > 0$ and sufficiently large $d$, the random variable $\mathbf{M}_k(d)$ satisfies the upper bound
   \begin{equation*}
        \mathbf{M}_k(d) < (4 + \eps) \cdot \frac{3}{2 \pi^2} \left(\frac{k}{k - 1} \right)^2 (\log d)^2.
    \end{equation*}
In particular, $\mathbf{M}_k(d) \asymp (\log d)^2$ as $d \to \infty$.
\medskip
\item For any fixed $d \geq 0$, the sequence of random variables $\{ \mathbf{M}_k(d) \}_{k=2}^{\infty}$ is bounded.
\end{enumerate}
\end{theorem}

More generally, we have analogous results to Theorem \ref{thm:introheuristic} for any function $f(n)$ that meets the conditions of our model.

\begin{remark}\label{rem:thm1.3} Note that part (a) of Theorem~\ref{thm:introheuristic} above gives a heuristic proof of Conjecture~\ref{conj:MOT1}, part (b) is in the direction of Sun's Conjecture, part (c) gives a heuristic proof of Conjecture~\ref{conj:Mkdasymp}, and part (d) gives a heuristic proof of Conjecture~\ref{conj:MOT3&4}.
\end{remark}

This paper is structured as follows.  In Section~\ref{sec:MOTconjs}, we introduce the notation $\Delta_{f,k}(n)$ and $M_{f,k}(d)$ for an arbitrary function $f : \N \to \N$, and study their asymptotic behavior.  In Section~\ref{sec:DMbounds}, we obtain bounds for $\Delta_{f,k}(n)$ and $M_{f,k}(d)$ and prove Theorem \ref{thm:Mkd_intro}.  In Section~\ref{sec:Nbounds} we obtain bounds for $N_{f,d}$ and prove Theorem~\ref{thm:Nd_intro}.  In Section~\ref{sec:equidistribution} we provide strong computational evidence that the sequence $\{ \{ \sqrt[k]{p(n)}\} \}_{n=1}^{\infty}$ of fractional parts of $\sqrt[k]{p(n)}$ is equidistributed modulo 1, which motivates our probabilistic model and heuristic results.

In Section~\ref{sec:model}, we introduce our probabilistic model and define a sequence of random variables $\{ \mathbf{f}_n \}_{n=1}^{\infty}$, used to model any sufficiently fast-growing function $f : \N \to \N$. We define the random variables $\mathbf{\Delta}_{\mathbf{f},k}(n)$ and $\mathbf{M}_{\mathbf{f},k}(d)$ analogous to $\Delta_{f,k}(n)$ and $M_{f,k}(d)$, and prove various bounds on $\mathbb{P}(\mathbf{\Delta}_{\mathbf{f},k}(n) \leq d )$.  In Section~\ref{sec:heuristics}, we prove a heuristic upper bound for $\mathbf{M}_{\mathbf{f},k}(d)$, and give a proof of Theorem~\ref{thm:introheuristic}.  In Section \ref{sec:computations}, we extend the computations of Merca--Ono--Tsai and provide further computational evidence towards Conjecture~\ref{conj:MOT1} both for $p(n)$ and various other partition functions.  Finally, in Section~\ref{sec:conjectures}, we provide further conjectures for other partition functions and study related functions $\widetilde{\Delta}_{f,a}(n)$ and $\widetilde{M}_{f,a}(d)$.

\section*{Acknowledgments}
This project was started at the \textit{Rethinking Number Theory 6} AIM workshop held online during June 2025. We owe a big thank you to the organizers of that workshop: Jen Berg, Heidi Goodson, and Allechar Serrano L\'{o}pez, for bringing all of us together to work on this project.

\section{Asymptotic behavior of $\Delta_{f,k}(n)$ and $M_{f,k}(d)$} \label{sec:MOTconjs}

Throughout this paper we let $f: \N \to \N$ denote a function such that $f(n)\rightarrow \infty$ as $n\rightarrow \infty$.  For any integers $k \geq 2$ and $n\geq 1$, let $\Delta_{f,k}(n)$ be the distance between $f(n)$ and its nearest $k$th power, namely
\begin{equation}\label{eq:deltadef}
    \Delta_{f,k}(n) := \min \left\{ | f(n) - m^k | : m \in \Z \right\} .
\end{equation}
We observe that $\Delta_{f,k}(n)$ can be computed as
\begin{equation} \label{eq:deltamin}
    \Delta_{f,k}(n) = \min \left( f(n) - \left\lfloor \sqrt[k]{f(n)} \right\rfloor^k \:,\: \left\lceil \sqrt[k]{f(n)} \right\rceil^k - f(n) \right),
\end{equation}
i.e., to compute $\Delta_{f,k}(n)$, it suffices to take the minimum of (i) the distance between $f(n)$ and the largest $k$th power less than or equal to $f(n)$, and (ii) the distance between $f(n)$ and the smallest $k$th power greater than or equal to $f(n)$.

\begin{remark}
If $k = 2$, then also $\Delta_{f,2}(n) = \left| f(n) - \left\lfloor \sqrt{f(n)} \, \right\rceil^2 \right|$, where $\lfloor m \rceil$ represents the closest integer function applied to $m$.  However this simplification does not generalize to $k > 2$. For example, if $f(n) = 4$ for some fixed $n$, then $\Delta_{f,3}(n) = 3$, but $\left| f(n) - \left\lfloor \sqrt[3]{f(n)} \right\rceil^3 \right| = 4$. 
\end{remark}

Since 1 is a $k$th power for all $k>1$, we have $\Delta_{f,k}(n)\leq |f(n)-1^k|=f(n)-1$ for all $k,n$. We show in the following proposition that $\{\Delta_{f,k}(n)\}_{k=1}^\infty$ in fact always stabilizes at that value.

\begin{prop} The sequence $\{\Delta_{f,k}(n)\}_{k=1}^\infty$ converges to $f(n)-1$.
\end{prop}

\begin{proof} Fix $n\in\N$. First note that $\Delta_{f,k}(n)$ is defined for all $k\geq 1$ by well-ordering. Find $K\in\N$ such that $f(n)\leq 2^{K-1}$. Then for all $k\geq K$, $2f(n)<2^K+1\leq 2^k+1$, so $f(n)-1<2^k-f(n)$, and hence 1 is the closest $k$th power to $f(n)$. Thus $\Delta_k(n)=f(n)-1$ for all $k\geq K$.
\end{proof}

Now fixing $k>1$, we define 
\begin{equation}\label{eq:Mdef}
    M_{f,k}(d) := \max \{ n \in \N : \Delta_{f,k}(n) \leq d \}
\end{equation}
if such an $n$ exists, and $\infty$ otherwise.  We note that Conjecture \ref{conj:MOT1} is equivalent to the statement that $M_{f,k}(d)$ is finite for all $k>1$ and $d\geq 0$.  

In order for $\{\Delta_{f,k}(n)\}_{n=1}^\infty$ to tend to $\infty$, it is necessary that $M_{f,k}(d)$ is finite for all $d\geq 0$.  Conversely, if $M_{f,k}(d)$ is finite for all $d\geq 0$, then $\{M_{f,k}(d)\}_{d=0}^\infty \rightarrow \infty$ and furthermore is monotonically increasing, since if $d\leq e$, then any $n$ satisfying $\Delta_k(n)\leq d$ will also satisfy $\Delta_k(n)\leq e$.

If we instead fix $d$ and let $k$ grow, we determine that if the sequence $\{ M_{f,k}(d) \}_{k=2}^{\infty}$ is bounded, then the limiting behavior of $M_{f,k}(d)$ matches that in Conjecture \ref{conj:MOT3&4}.  Generalizing \eqref{eq:Lp}, define
\begin{equation}\label{eq:Lf}
L_f(d) = \max \{n: f(n) \leq d + 1\}
\end{equation}
if such an $n$ exists, and $\infty$ otherwise.  From the definition we observe that $L_f(d)$ is monotonically increasing in $d$.

\begin{prop} \label{prop:limit}  
    If the sequence $\{ M_{f,k}(d) \}_{k=2}^{\infty}$ is bounded, then  
    \begin{equation*}
        \lim_{k\to \infty} M_{f,k}(d) = L_f(d).
    \end{equation*}
\end{prop}

\begin{proof}
    Since $f(n)$ grows with $n$, $L_f(d)$ must be finite.  We first show that $M_{f,k}(d) \geq L_f(d)$ for all $k$. Note that we have $f(L_f(d)) - 1 \leq d$. Now for any $k \geq 2$, as $1$ is clearly a $k$th power, this implies $\Delta_{f,k}(L_f(d)) \leq f(L_f(d)) - 1 \leq d$. and therefore by definition $M_{f,k}(d) \geq L_f(d)$ for all $k \geq 2$.

    Now let $A$ be the smallest integer such that $M_{f,k}(d) \leq A$ for all sufficiently large $k$.  Note that $A$ exists as $\{M_{f,k}(d)\}_{k=1}^{\infty}$ bounded (by assumption), and $A \geq L_f(d)$ (as $M_{f,k}(d) \geq L_f(d)$ for all $k$).  We shall prove that $A = L_f(d)$.

    Let us assume for contradiction that $A > L_f(d)$.  By definition of $A$, there are infinitely many integers $k$ such that $M_{f,k}(d) = A$.   In particular, we can choose some $K$ sufficiently large such that  $M_{f,K}(d) = A$ and $2^K > f(A) + d$. We now claim that $\Delta_{f,K}(A) > d$.   As $A > L_f(d)$, this implies $f(A) > 1 + d$, and as $2^K > f(A) + d$, this implies that $| f(A) - m^K | > d$ for all $m \in \Z$. Therefore
     \begin{equation*}
        \Delta_{f,K}(A) := \min \{ | f(A) - m^K | : m \in \Z \} > d ,
    \end{equation*}
    but, $M_{f,K}(d) = A$ also implies that $\Delta_{f,K}(A) \leq d$, which is a contradiction!

    Thus $A = L_f(d)$, which implies $M_{f,k}(d) = L_f(d)$ for all sufficiently large $k$. Thus $\lim_{k \to \infty} M_{f,k}(d)$ exists and equals $L_f(d) = \max \{ n \in \N : f(n) \leq d + 1 \}$.
\end{proof}

Progress toward bounding $\{M_{f,k}(d)\}_{k\geq 2}$ in general is given by the following result.

\begin{prop} \label{prop:multiple}
Let $f:\N\to\N$ and fix $d\geq 0$ and $k\in\N$. Then $M_{f,k}(d)\geq M_{f,jk}(d)$ for all $j\in\N$.
\end{prop}

\begin{proof} Let $j\in\N$ and suppose that $M_{f,jk}(d)=n$. Then for some $m\in\Z$ we have $|f(n)-m^{jk}|\leq d$. So, also, $|f(n)-(m^j)^k|\leq d$, hence $M_{f,k}(d)\geq n=M_{f,jk}(d)$.
\end{proof}

By Proposition \ref{prop:multiple}, it is sufficient to bound $\{M_{f,p_j}(d)\}_{j\geq 1}$ for all primes $p_j\in\N$ in order to show that  $\{M_{f,k}(d)\}_{k\geq 2}$ is bounded.  

\begin{remark} For the partition function $p(n)$, our computations suggest that $M_2(d)\geq M_p(d)$ for all primes $p$. If this could be proved, then Conjecture \ref{conj:sun} would follow by proving that $M_2(0)=1$.  One possible approach to proving that $M_2(d)$ is an upper bound would be to obtain sufficiently precise asymptotics for the growth of $M_{k}(d)$ as $d \rightarrow \infty$ for each fixed $k$. \end{remark}

\begin{remark}  We note there are sequences where, for some fixed $d \geq 0$,  $M_{f,k}(d)$ is finite for all $k \geq 2$, but the sequence $\{ M_{f,k}(d) \}_{k=2}^{\infty}$ is unbounded.  For example, consider the sequence $f(n) = 2^{p_n}$ where $p_n$ is the $n$-th prime.  Then, for any $n \geq 1$, we have $M_{f,p_n}(0) = n$.  Furthermore, by Mihăilescu's proof of the Catalan conjecture \cite{Mihuailescu}, we know that the only two consecutive perfect powers are 8 and 9, and thus we have that $M_{f,k}(1)$ is finite for all $k \geq 2$.

More generally, Pillai's conjecture would imply that, given any $d \geq 0$, there are only finitely many pairs of perfect powers whose difference is at most $d$ (e.g. see \cite[Problem~D9]{Guy_book}).  This would therefore imply that $M_{f,k}(d)$ is finite for all $d$ and for all $k$.\end{remark}

\section{Bounds for $\Delta_{f,k}(n)$ and $M_{f,k}(d)$} \label{sec:DMbounds}

In this section we prove upper bounds for $\Delta_{f,k}(n)$ and lower bounds for $M_{f,k}(d)$ for fixed $k$ as $d\to\infty$. 

\begin{lemma} \label{lem:Delta_bound}
    Fix $k \geq 2$. For  $n\geq 1$, we have the upper bound
    \begin{equation*}
        \Delta_{f,k}(n) \leq \frac{1}{2} \sum_{i=1}^k  \binom{k}{i} f(n)^{(k-i)/k}.
    \end{equation*}
    Then for $d\geq 1$, we have the lower bound
     \begin{equation*}
        M_{f,k}(d) \geq \max \left\{  n \in \N \;\middle|\; \frac{1}{2} \sum_{i=1}^k  \binom{k}{i} f(n)^{(k-i)/k} \leq  d \right\} .
    \end{equation*}    
\end{lemma}

\begin{proof}
    Note that by (\ref{eq:deltamin}), $\Delta_{f,k}(n)$ can be expressed as the distance from $f(n)$ to either $\floor{ \sqrt[k]{f(n)} }^k$ or $\ceil{ \sqrt[k]{f(n)} }^k$, whichever is closer to $f(n)$.  Therefore, by taking half of the distance between $\floor{ \sqrt[k]{f(n)} }^k$ and $\ceil{ \sqrt[k]{f(n)} }^k$, this gives us the following upper bound for $\Delta_{f,k}(n)$:
    \begin{align*}
        \Delta_{f,k}(n) \leq \frac{1}{2} \left( \ceil{ \sqrt[k]{f(n)} }^k -  \floor{ \sqrt[k]{f(n)} }^k \right) \leq \frac{1}{2} \left( \left( \floor{ \sqrt[k]{f(n)} } + 1 \right)^k -  \floor{ \sqrt[k]{f(n)} }^k \right) .
    \end{align*}
    Now as $(x+1)^k - x^k$ is an increasing sequence in $x$, by an application of the binomial theorem, we have
    \begin{align*}
        \Delta_{f,k}(n) \leq \frac{1}{2} \left( \left(  \sqrt[k]{f(n)}  + 1 \right)^k -   \sqrt[k]{f(n)}^k \right) = \frac{1}{2} \sum_{i=1}^k  \binom{k}{i} f(n)^{(k-i)/k}.
    \end{align*}
    Then by the definition of $M_{f,k}(d)$, the given lower bound follows.
\end{proof}

For sufficiently large $n$, we can simplify the bound for $\Delta_{f,k}(n)$ in Lemma \ref{lem:Delta_bound} as follows. 

\begin{prop}\label{prop:Delta_betterbound}
    Let $k \geq 2$ and let $\eps > 0$.  Then for all sufficiently large $n$ we have
    \begin{equation*}
        \Delta_{f,k}(n) \leq (1 + \eps) \cdot \frac{k}{2} f(n)^{(k-1)/k} .
    \end{equation*}
\end{prop}

\begin{proof}
    This follows by expanding the binomial sum $\frac{1}{2}  \sum_{i=1}^k  \binom{k}{i} f(n)^{(k-i)/k} $ in Lemma~\ref{lem:Delta_bound} to obtain
    \begin{equation*}
        \frac{k}{2} f(n)^{(k-1)/k} + \frac{k(k-1)}{4} f(n)^{(k-2)/k} + \dots + \frac{k}{2} f(n)^{1/k} + 
        \frac{1}{2} ,
    \end{equation*}
    which has leading term $\frac{k}{2} f(n)^{(k-1)/k}$.  Then since $f(n)\to\infty$ we have that
    \begin{equation} \label{eq:epsbound}
        \frac{1}{2}  \sum_{i=1}^k  \binom{k}{i} f(n)^{(k-i)/k} \leq (1 + \eps) \cdot \frac{k}{2} f(n)^{(k-1)/k}
    \end{equation}
    for all sufficiently large $n$.  
\end{proof}

For sufficiently large $d$ we have a corresponding lower bound for $M_{f,k}(d)$.

\begin{prop} \label{prop:Mkdlower}
    Let $k \geq 2$ and let $0<\eps <1$.  Then, for all sufficiently large  $d$, we have
    \begin{equation*}
        M_{f,k}(d) \geq \max \left\{    n \in \N \;\middle|\; f(n) \leq (1 - \eps) \cdot \left( \frac{2d}{k} \right)^{k/(k-1)} \right\} .
    \end{equation*}
    
\end{prop}

\begin{proof}
      Let $\eps^{\prime} := (1/(1 - \eps))^{k/(k - 1)}-1$, and choose $d$ large enough that the largest $n$ satisfying $\frac{1}{2}  \sum_{i=1}^k  \binom{k}{i} f(n)^{(k-i)/k} \leq d$ also satisfies the inequality (\ref{eq:epsbound}) with $\eps'$.  Then by Lemma \ref{lem:Delta_bound}, for all such $d$ we have
    \begin{align*}
        M_{f,k}(d) &\geq \max \left\{  n \in \N \;\middle|\; \frac{1}{2} \sum_{i=1}^k  \binom{k}{i} f(n)^{(k-i)/k} \leq  d \right\} \\
        &\geq \max \left\{  n \in \N \;\middle|\; (1 + \eps') \cdot \frac{k}{2} f(n)^{(k-1)/k} \leq  d \right\} \\
        &= \max \left\{  n \in \N \;\middle|\;  f(n) \leq (1 - \eps) \cdot \left( \frac{2d}{k} \right)^{k/(k-1)} \right\}.
    \end{align*}
\end{proof}

In cases where $f(n)$ grows asymptotically in a similar way to $p(n)$, we can make the lower bound in Proposition \ref{prop:Mkdlower} explicit.

\begin{theorem}\label{asy_Mk}
    Let $f :\N \to  \N$ be a function such that $f(n) \sim an^{-b} \exp(\alpha n^{c})$ for some constants $a > 0, b \in \R, c > 0$ and $\alpha > 0$. Then for any fixed $k \geq 2$, we have
    \[
    M_{f,k}(d) \geq \left(\frac{k}{\alpha(k - 1)} \right)^{\frac{1}{c}} (\log d)^{\frac{1}{c}} \left( 1 + \frac{O(\log \log d)}{\log d} \right).
    \]
\end{theorem}

\begin{proof} Fix $k\geq 2$ and $\eps\in(0,1)$. By Proposition \ref{prop:Mkdlower}, a lower bound for $M_{f,k}(d)$ is given by the the largest $n$ such that
\begin{equation} f(n)\leq(1-\eps)\left(\frac{2d}{k}\right)^\frac{k}{(k-1)}.\label{eq:fbound}\end{equation}
Since $f(n) \sim an^{-b} \exp(\alpha n^{c})$, then for any fixed $\eta\in(0,1)$ there exists $N_\eta$ such that for all $n\geq N_\eta$,
\begin{equation*}f(n)\leq (1+\eta)an^{-b} \exp({\alpha n^c}).\end{equation*}
Thus we guarantee (\ref{eq:fbound}) by requiring
\begin{equation*}
    (1+\eta)an^{-b} \exp({\alpha n^c}) \leq (1-\eps)\left(\frac{2d}{k}\right)^{k/(k-1)},
\end{equation*}
or equivalently
\begin{equation*}
    n^{-b} \exp({\alpha n^c}) \leq C_0d^{k/(k-1)}
\end{equation*}
for some constant $C_0>0$. Taking logarithms yields
\begin{equation}\label{eq:n_ineq}
\alpha n^c\leq \frac{k}{k-1}\log d+b\log n+O(1).
\end{equation}
Now we solve this inequality asymptotically as $d\to\infty$ by proposing that $n$ has the form 
\begin{equation*}
    n=\left(\frac{k}{\alpha(k-1)}\right)^{1/c}(\log d)^{1/c}(1+\beta_d)
    \end{equation*}
    where $\beta_d=o(1)$ is a function to be chosen later. We have
\begin{equation*}
    \alpha n^c=\frac{k}{k-1}(\log d)(1+c\beta_d+O(\beta_d^2))=\frac{k}{k-1}\log d+\frac{k}{k-1}c\beta_d\log d+O(\beta_d^2\log d),
\end{equation*}
and
\begin{equation*}
    b\log n=\frac{b}{c}\log\log d+O(1)+O(\beta_d).
\end{equation*}
Substituting these values in (\ref{eq:n_ineq}), we have
\begin{equation*}
    \frac{k}{k-1}c\beta_d\log d\leq \frac{b}{c}\log\log d+O(1)+O(\beta_d)+O(\beta_d^2\log d).
\end{equation*}
Now we see that we can choose $\beta_d=C_1\frac{\log\log d}{\log d}$ for some constant $C_1>0$ so that $\beta_d$ and $\beta_d^2\log d$ are each $o(\log\log d)$ and the inequality is satisfied for sufficiently large $d$.
\end{proof}

In particular, applying Theorem \ref{asy_Mk} to $f(n)=p(n)$ and using the Hardy-Ramanujan asymptotic proves Theorem \ref{thm:Mkd_intro}.

\section{Lower bounds for $N_{f,d}$}\label{sec:Nbounds}

In this section we consider $f$ such that for any $d$, 
\begin{equation}\label{eq:Nhypothesis}
\{M_{f,k}(d)\}_{k=2}^\infty \rightarrow L_f(d),
\end{equation}
where $L_f(d)$ is as defined in \eqref{eq:Lf}. In particular, $L_f(d)$ is monotonically increasing in $d$.  Note that if $f(n) \rightarrow \infty$ as $n \rightarrow \infty$, then it follows that $L_f(d) \rightarrow \infty$ as $d \rightarrow \infty$.

Under hypothesis \eqref{eq:Nhypothesis}, we define the constant $N_{f,d}$ by
\begin{equation}\label{eq:Ndef}
    N_{f,d} = \min \{ n \in \N : M_{f, k}(d) = L_f(d) \text{ for all } k \geq n \},
\end{equation}
i.e., $N_{f,d}$ is the smallest value of $k$ past which $\{M_{f,k}(d)\}_{k=1}^{\infty}$ is stable at $L_f(d)$.

Merca--Ono--Tsai \cite[p.~104]{MercaOnoTsai} observed that, for the partition function $p(n)$, the value of $N_d$ appears to grow like $O(\log{d})$.  In the following, we prove a logarithmic lower bound for functions $f(n)$ with sufficiently controlled growth. 

\begin{theorem} \label{thm:ndlower}
    If $f(n)$ satisfies \eqref{eq:Nhypothesis} and $f(n+1) < \frac{3}{2} f(n)$ for all sufficiently large $n$, then $$N_{f,d} \geq \lfloor \log_2(d) \rfloor + 1$$ for all sufficiently large $d$. 
\end{theorem}

\begin{proof}
    Since $f(n) \rightarrow \infty$ we have $L_f(d) \rightarrow \infty$ as $d \rightarrow \infty$.  Thus for sufficiently large $d$, we have 
    \begin{equation}\label{eq:star}
    f(L_f(d)+1) < \frac{3}{2} f(L_f(d)) \leq \frac{3}{2} (d+1)
    \end{equation} by our hypothesis.  Since $\{M_{f,k}(d)\}_{k=2}^\infty \rightarrow L_f(d)$, to prove the claim it suffices to show that $M_{f,k}(d) > L_f(d)$ for $k = \floor{\log_2(d)}$.
    
By definition of $L_f(d)$, it follows that $f(L_f(d) + 1) > d + 1$.  Fixing $k = \floor{\log_2(d)}$, we have $d \geq 2^k$, and therefore $f(L_f(d)+1) - 2^k > 0$.  We also have that $\frac{d}{2} < 2^k$, and by considering the parity of $d$ it follows that $\frac{d+1}{2} \leq 2^k$.   Therefore from \eqref{eq:star} we have
    \begin{equation*}
       0 <  f(L_f(d) + 1) - 2^k < \frac{3}{2}(d+1) - \frac{d+1}{2} = d + 1 .
    \end{equation*}
    It follows that  $|f(L_f(d) + 1) - 2^k | < d+1 $, and so $|f(L_f(d) + 1) - 2^k | \leq d$.  Hence by \eqref{eq:deltadef}, $\Delta_{f,k}(L_f(d)+1) \leq d$, which implies $M_{f,k}(d) > L_f(d)$.  This proves the claim.
    
\end{proof}

If $f(n)$ has the property that $f(n) \sim an^{-b} \exp( \alpha n^c)$ for $\alpha > 0$ and $c \in (0, 1)$, then $f$ satisfies the assumptions of Theorem~\ref{thm:ndlower}.  In this case we prove a more refined bound for $N_{f,d}$. 

\begin{theorem}\label{thm:ndlower2terms}
    Suppose $f(n)$ satisfies \eqref{eq:Nhypothesis} and also that $f(n) \sim an^{-b} \exp( \alpha n^c)$ 
    as $n \to \infty$, for some $\alpha > 0$ and $c \in (0, 1)$.
    Then for any constant $0 < A < (1 - c)/(c \log 2)$, we have 
    \[
    N_{f,d} > \log_{2}(d) + A \log \log d 
    \]
    for all sufficiently large $d$. 
\end{theorem}

\begin{proof}
From our hypotheses we have that for every $\epsilon > 0$, there exists $N_{\epsilon}$ such that for all $n \geq N_{\epsilon}$, 
\[
(\alpha - \epsilon) n^c \leq \log f(n) \leq (\alpha + \epsilon) n^c.
\]
Thus for $n \geq N_{\epsilon}$,
\begin{equation}\label{eq:hypbds}
\exp((\alpha - \epsilon) n^c) \leq f(n) \leq \exp((\alpha + \epsilon) n^c).
\end{equation}
Fix $\epsilon \in (0, \alpha)$. Since $f(n) \rightarrow \infty$ we have $L_f(d) \rightarrow \infty$, so for sufficiently large $d$ we ensure that $L_d \geq N_{\epsilon}$. Using \eqref{eq:hypbds} and the fact that $f(L_f(d)) \leq d + 1$ by definition, we have
\[
\exp((\alpha - \epsilon) L_f(d)^c) \leq f(L_f(d)) \leq d + 1.
\]
Thus,
\[
L_f(d)^c \leq \frac{\log (d + 1)}{\alpha - \epsilon},
\]
and setting $C_2 = \left(\frac{1}{\alpha - \epsilon}\right)^{1/c}$, we have for sufficiently large $d$,
\begin{equation}\label{eq:LbdC2}
L_f(d) \leq C_{2} (\log d)^{1/c}.
\end{equation}
Further set $C_1 = \left(\frac{1}{\alpha + \epsilon} \right)^{1/c}$.  Considering $n_0 = C_1 (\log(d + 1))^{1/c}$, \eqref{eq:hypbds} implies
\[
\log f(n_0) \leq (\alpha + \epsilon) n_0^c \leq \log (d + 1),
\]
and thus $f(n_0) \leq d + 1$ and $n_0 \leq L_f(d)$.  Hence, $L_f(d) \geq C_1 (\log(d + 1))^{1/c}$, and with \eqref{eq:LbdC2} we obtain that for sufficiently large $d$,
\begin{equation}\label{asy_Ld}
    C_{1} (\log d)^{1/c} <  L_f(d) \leq C_{2} (\log d)^{1/c}.
\end{equation}

We now derive an asymptotic for the difference $g(n) := f(n + 1) - f(n)$.  By the Maclaurin series expansion for $e^{-x}$, we note that we have the asymptotic $1 - e^{-x} \sim x$ as $x \to 0$.  Following a similar strategy as done in \cite[Theorem~1]{AGHHPSS}, the asymptotic for the ratio $g(n) / f(n)$ is given by
\begin{equation*}
    \frac{g(n)}{f(n)} =  \frac{f(n + 1)}{ f(n) }  - 1 
    \sim \exp{\left( b \log \Big( \frac{n}{n + 1} \Big) + \alpha  \left(  (n+1)^c - n^c \right) \right)} - 1. 
\end{equation*}
Since $1 - e^{x} \sim -x \text{ as } x \to 0$, we obtain $\tfrac{g(n)}{f(n)} \sim \alpha c n^{c-1}$, using that facts that $\log\!\bigl(\tfrac{n}{n + 1}\bigr) = - \tfrac{1}{n} + O(n^{-2})$, and $(n + 1)^c - n^c = cn^{c-1} + O(n^{c-2})$.
Therefore, we have the asymptotic
    \begin{equation} \label{eq:diff_asymp}
        g(n) \sim f(n) \alpha c n^{c-1}
    \end{equation}
    as $n \to \infty$.

Using \eqref{eq:diff_asymp}, there exists $N_1 > 0$ such that for all $n \geq N_1$, we have
\begin{equation}\label{f_inc}
B_{1} f(n) n^{c - 1} \leq f(n + 1) - f(n) \leq B_{2} f(n) n^{c - 1},
\end{equation}
for some two positive constants $B_{1}, B_2 > 0$.

For $d$ large enough to ensure $L_f(d) \geq \max \{N_\epsilon, N_1\}$, define
\begin{equation}\label{eq:nddef}
n_{d}:= \max \{n \geq L_f(d) : f(j + 1) - f(j) \leq 2d  \text{ for all } L_f(d) \leq j \leq n \}.
\end{equation}
Using \eqref{asy_Ld} and \eqref{f_inc}, it follows that the above set is non-empty at $j = L_f(d)$ for sufficiently large $d$.  Moreover, $n_{d}$ is finite since \eqref{eq:hypbds} and \eqref{f_inc} together imply that $f(n + 1) - f(n) \rightarrow \infty$ as $n \rightarrow \infty$.  By the maximality of $n_{d}$, we have that $f(n_d + 2) - f(n_d + 1) > 2d$, and for any $L_{f, d} \leq n \leq n_d$,
\begin{equation}\label{bound_f}
f(n + 1) - f(n) \leq 2d. 
\end{equation}
Note that \eqref{bound_f} implies that any real $y\in [f(L_f(d) + 1), f(n_d + 1)]$ can be at most $d$ away from the nearest $f(j)$ value.  In particular, for any such $y$ there exists an integer $s_y\in [L_f(d) + 1 , n_d + 1]$ such that 
\begin{equation}\label{ext_s}
|f(s_y) - y| \leq d. 
\end{equation}

Moreover, since $f(n_d + 2) - f(n_d + 1) > 2d$, using \eqref{f_inc} gives
\[
2d < f(n_d + 2) - f(n_d + 1) \leq B_2 f(n_d + 1) (n_d + 1)^{c - 1}.
\]
This implies that
\begin{equation} \label{lower_bound_f}
    f(n_d + 1) > \frac{2d}{B_2} (n_d + 1)^{1 - c}, 
\end{equation}
and thus 
\begin{equation}\label{eq:1}
    \log f(n_d + 1) > \log d + (1 - c) \log (n_d + 1) + O(1).
\end{equation}
Since $n_d \geq L_f(d)$, \eqref{asy_Ld} implies that for $d$ sufficiently large, $n_d > C_1 (\log d)^{1/c}$. Therefore,
\begin{equation} \label{eq:2}
    \log (n_d + 1) > \log C_1 + c^{-1} \log \log d.
\end{equation}
Substituting the value of \eqref{eq:2} in \eqref{eq:1}, we obtain that
\begin{equation*}
\log f(n_d + 1)  > \log d + \frac{1 - c}{c} \log \log d + O(1).
\end{equation*}
Therefore, we have
\begin{equation} \label{eq:3}
    \log_2 f(n_d + 1) = \frac{\log f(n_d + 1)}{\log 2} >  \log_2 d + \frac{1 - c}{c \log 2} \log \log d + O(1).
\end{equation}
We define
\begin{equation}\label{eq:kddef}
k_d := \lfloor \log_2 f(n_d + 1)   \rfloor,
\end{equation}
which immediately gives 
\begin{equation}\label{eq:trivbd}
2^{k_d} \leq f(n_d + 1) < 2^{k_d + 1}. 
\end{equation}
From \eqref{eq:3} we observe that for any fixed $0 < A < \frac{1 - c}{c \log 2}$, we have for all sufficiently large $d$, 
\begin{equation}\label{eq:fndbds}
\log_2 f(n_d + 1) > \log_2 d + A \log \log d.
\end{equation}
From \eqref{eq:kddef} and \eqref{eq:fndbds}, we deduce that $k_d \geq \lfloor \log_2 d + A \log \log d \rfloor$ for all sufficiently large $d$. Further, since $\frac{f(n_d + 1)}{f(L_f(d) + 1)} \rightarrow \infty$ as $d \rightarrow \infty$, 
we have $f(n_d + 1) \geq 2 f(L_f(d) + 1)$ for sufficiently large $d$. With \eqref{eq:trivbd} this gives
\[
f(L_f(d) + 1) \leq \frac{1}{2} f(n_d + 1) < 2^{k_d} \leq f(n_d + 1).
\]

Thus by \eqref{ext_s}, there exists an integer $s$ with $L_f(d) + 1 \leq s \leq n_d + 1$ such that $|f(s) - 2^{k_d}| \leq d$. Thus with \eqref{eq:deltadef} and \eqref{eq:Mdef} we have 
\begin{equation*}
M_{f, k_d} \geq s > L_d.
\end{equation*}
From \eqref{eq:Ndef}, we conclude that $N_{f,d}>k_d$.  Thus using \eqref{eq:kddef} and \eqref{eq:fndbds}, we obtain
\[
N_{f,d} \geq k_d+1 > \log_2 f(n_d+1) \geq  \log_2 d  + A \log \log d
\]
for sufficiently large $d$. 
\end{proof}

Since \eqref{eq:HRasymptotic} guarantees that $p(n)$ satisfies the growth condition in the hypotheses of Theorem \ref{thm:ndlower2terms}, then Proposition \ref{prop:limit} and Theorem \ref{thm:ndlower2terms} together prove Theorem \ref{thm:Nd_intro}.

\section{Equidistribution} \label{sec:equidistribution}

Here we present strong supporting evidence for the following conjecture.  Let $\{ r \}$ denote the fractional part of a real number $r$.

\begin{conj}\label{conj: equidistribution} For each fixed integer $k\geq 2$, the sequence of the fractional parts of $p(n)^{1/k}$, 
$$\{ \{ \sqrt[k]{p(n)}\} \}_{n\geq 1},$$
is equidistributed in $[0,1)$. \end{conj}

We have computed the sequence for $1\leq n\leq N$ with $N= 5000$ and $N= 500,000$ and tested equidistribution using several complementary diagnostics: histograms, empirical distribution functions, and discrepancy estimates. 

\begin{figure}[h!]
    \centering

    \begin{subfigure}{0.49\linewidth}
        \centering
        \includegraphics[width=\linewidth]{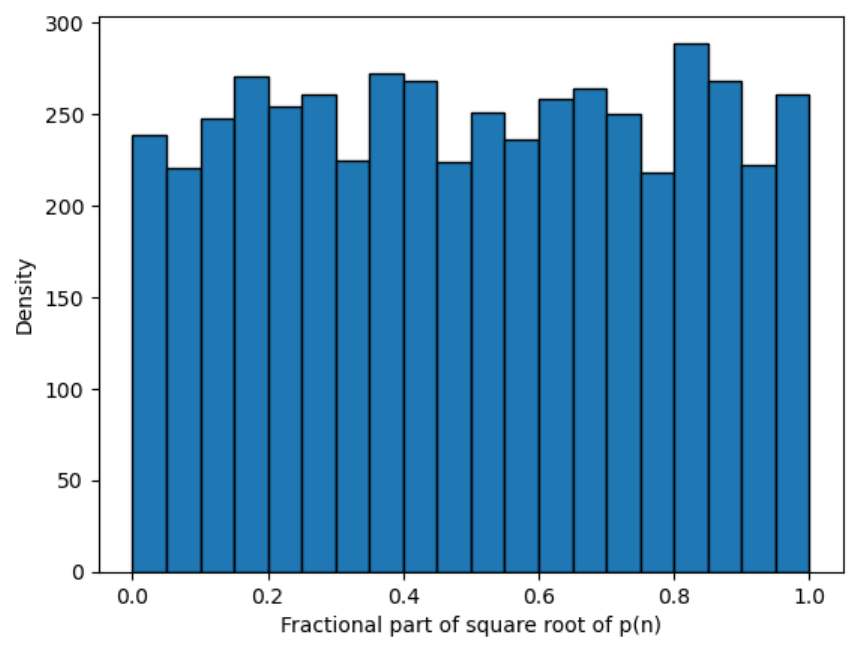}
        \caption{$k=2$, $n\leq 5000$}
        \label{fig:histogram k=2 5k}
    \end{subfigure}
    \hfill
    \begin{subfigure}{0.49\linewidth}
        \centering
        \includegraphics[width=\linewidth]{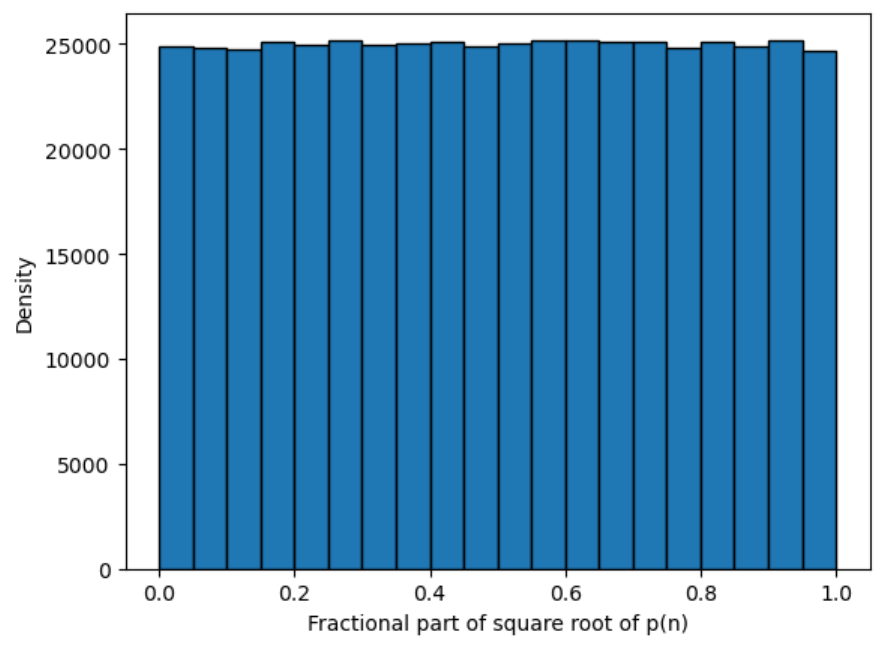}
        \caption{$k=2$, $n\leq 500\,000$}
        \label{fig:histogram k=2 500k}
    \end{subfigure}

    \vspace{1em}

    \begin{subfigure}{0.49\linewidth}
        \centering
        \includegraphics[width=\linewidth]{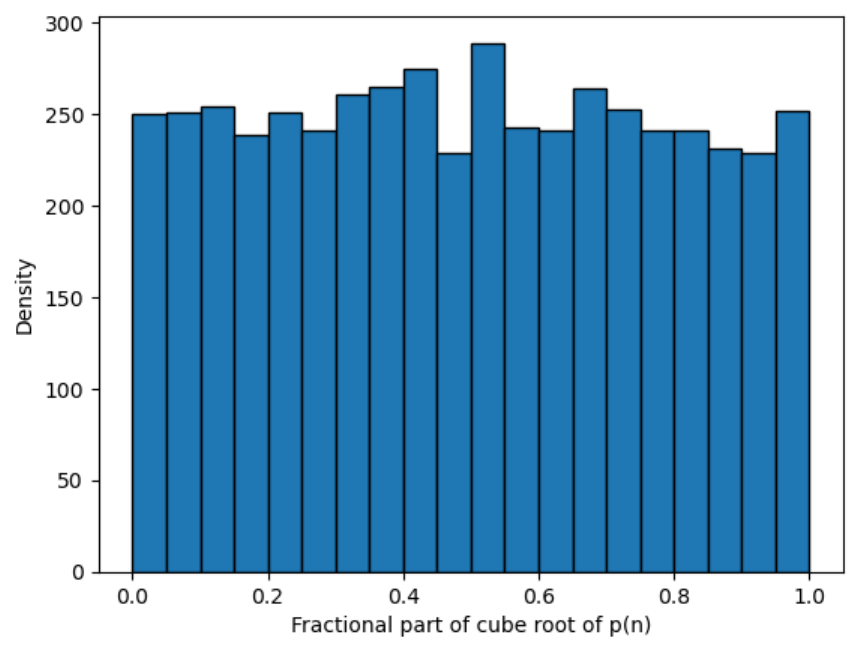}
        \caption{$k=3$, $n\leq 5000$}
        \label{fig:histogram k=3 5k}
    \end{subfigure}
    \hfill
    \begin{subfigure}{0.49\linewidth}
        \centering
        \includegraphics[width=\linewidth]{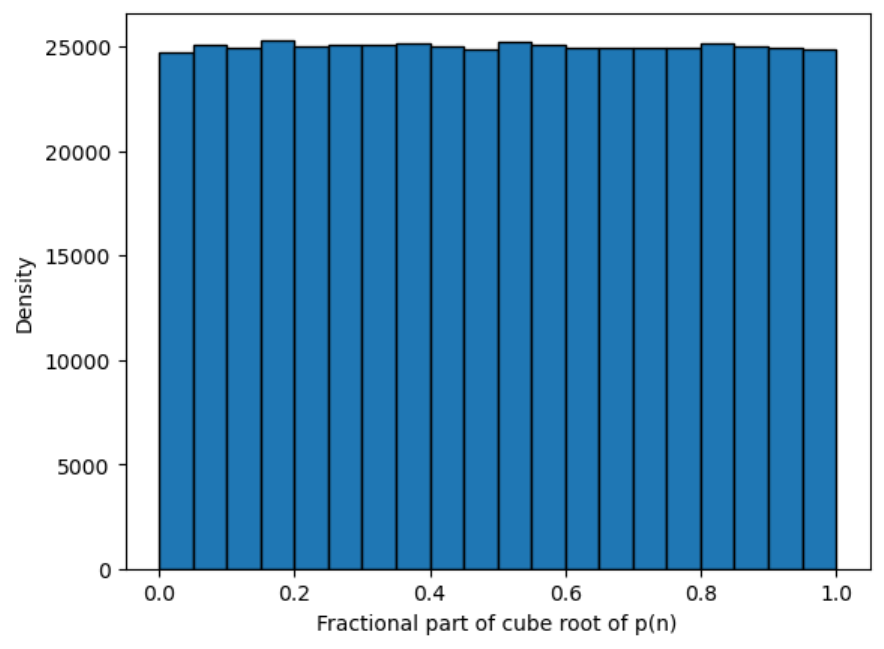}
        \caption{$k=3$, $n\leq 500\,000$}
        \label{fig:histogram k=3 500k}
    \end{subfigure}

\vspace{1em}

    \begin{subfigure}{0.49\linewidth}
        \centering
        \includegraphics[width=\linewidth]{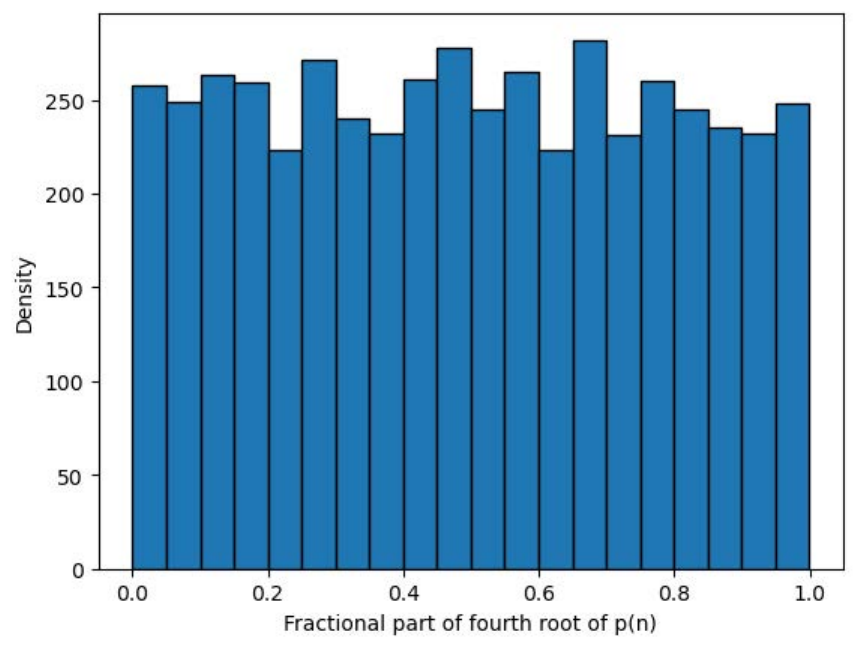}
        \caption{$k=4$, $n\leq 5000$}
        \label{fig:histogram k=4 5k}
    \end{subfigure}
    \hfill
    \begin{subfigure}{0.49\linewidth}
        \centering
        \includegraphics[width=\linewidth]{frrt_500000.pdf}
        \caption{$k=4$, $n\leq 500\,000$}
        \label{fig:histogram k=4 500k}
    \end{subfigure}

    \caption{Histograms showing distribution of fractional parts of $p(n)^{1/k}$ for $k=2,3,4$ and increasing values of $n$.}
    \label{fig:histogram}
\end{figure}

Figure~\ref{fig:histogram} shows histograms of fractional parts of $p(n)^{1/k}$ for $k=2,3,4$ and $n\leq N$, with $N=5000$ and $N=500,000$. 
We notice, as $N$ increases, the histograms become progressively flatter, suggesting that the distribution of fractional parts stabilizes toward equidistribution. 
While histogram-based diagnostics depend on the choice of subintervals, we observed this stabilization for increasing $N$ across a wide range of sizes of subintervals, providing consistent numerical evidence for equidistribution.

To remove this dependence and to obtain a more quantitative picture, we next plot the empirical distribution function (EDF) of the fractional parts of $p(n)^{1/k}$, given by 
$$F_N(x)=\frac{\#\{n\leq N\ |\ \{p(n)^{1/k}\}\leq x \}}{N},$$
against the diagonal corresponding to the uniform distribution on $[0,1)$.

\begin{figure}[h!]
    \centering

    \begin{subfigure}{0.49\linewidth}
        \centering
        \includegraphics[width=\linewidth]{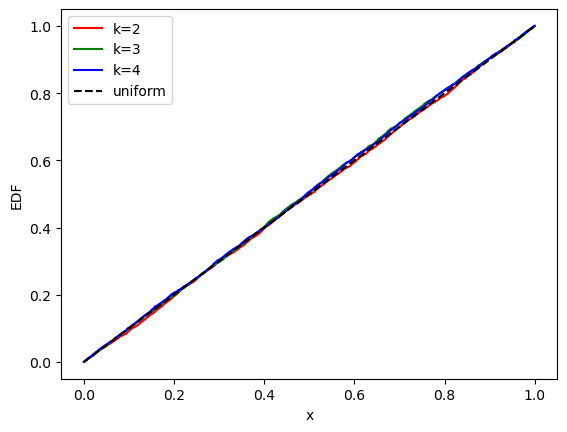}
        \caption{$n \leq 5000$}
        \label{fig:edf 5k}
    \end{subfigure}
    \hfill
    \begin{subfigure}{0.49\linewidth}
        \centering
        \includegraphics[width=\linewidth]{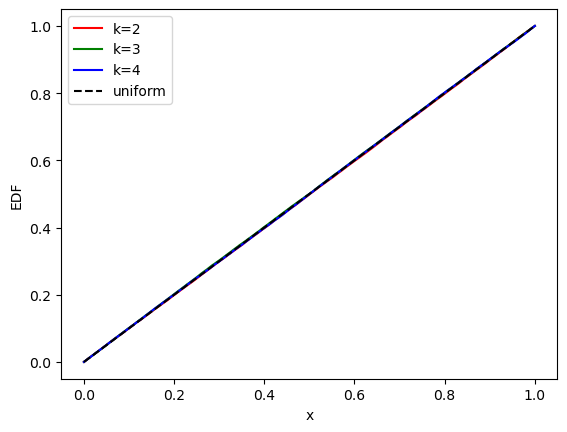}
        \caption{$n \leq 500\,000$}
        \label{fig:edf 500k}
    \end{subfigure}

    \caption{Empirical distribution functions of fractional parts of $p(n)^{1/k}$ for $k=2,3,4$ and increasing values of $n$, compared with the uniform distribution on $[0,1)$.}
    \label{fig:edf}
\end{figure}

Figure~\ref{fig:edf} compares the empirical distribution functions of the fractional parts of $p(n)^{1/k}$ for $k=2,3,4$ and $n\leq N$ with $N= 5000$ and $N= 500,000$ with the uniform distribution. 
We observe that the empirical distribution functions closely track the diagonal $y=x$, and the deviations appear to decrease significantly as $N$ increases.

We next compute the Kolmogorov–Smirnov statistic
$$D_N= \max_{x\in [0,1)} |F_N(x)-x|. $$
This discrepancy measures the maximal deviation of the EDF from the uniform distribution and provides a quantitative complement to the visual EDF plots.

\begin{figure}[htbp]
    \centering
    \includegraphics[width=.8\textwidth]{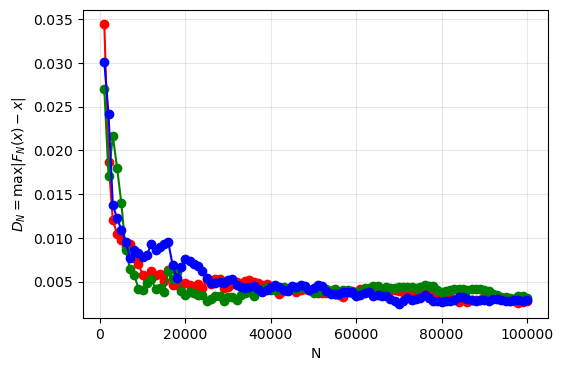}
    \caption{Kolmogorov–Smirnov statistics for fractional parts of $p(n)^{1/k}$ for $k=2,3,4$, and $N=100\,000$.}
        \label{fig:ks stat}
\end{figure}

Figure~\ref{fig:ks stat} shows $D_N$ as a function of $N$ for $k=2,3,4$. We see that $D_N$ decreases as $N$ grows, consistent with equidistribution.

Taken together, the histogram plots, empirical distribution functions, and discrepancy statistics provide strong numerical evidence in support of Conjecture \ref{conj: equidistribution}. We emphasize that these plots are purely experimental and do not constitute proof.

\section{A probabilistic model}\label{sec:model}

Throughout Sections \ref{sec:model} and \ref{sec:heuristics}, all random variables are written in boldface  (e.g. $\mathbf{A}, \mathbf{B}, \mathbf{C}$). We use $\mathbb{P} E$ to denote the probability that some event $E$ occurs, and $\mathbb{E} \mathbf{X}$ to denote the expected value of some random variable $\mathbf{X}$. We shall also use Vinogradov’s notation $f \ll g$ to mean $f = O(g)$. 

\subsection{Setup of the heuristic model}\label{ss:setup}
Here we give a precise description of our probabilistic model, motivated by the evidence for Conjecture~\ref{conj: equidistribution} given in Section~\ref{sec:equidistribution}.  Let $A : \N \to \R^+$ be a positive real-valued function on the natural numbers $\N$. 
We impose the growth condition that $A(n) \to \infty$ as $n \to \infty$, and moreover that $\sum_{n=1}^{\infty} A(n)^{-\frac12}$ converges. To model $f(n)$ we will set $A(n)$ to be the main term of the asymptotic for $f(n)$.

Let $\{\eps_n\}_{n=1}^\infty$ be a sequence of positive real numbers such that $\eps_n \to 0$ as $n \to \infty$.  To control the rate at which $\eps_n$ converges to 0, we also assume that as $n \to \infty$,

\begin{equation}\label{eq:eps_ncondition}
\eps_n \gg \frac{\log A(n)}{\sqrt{A(n)}}.    
\end{equation}

Let $\{\mathbf{f}_n\}_{n=1}^\infty$ be a sequence of discrete random variables defined so that for each $n$, $\mathbf{f}_n$ takes a uniform distribution among the set 
\begin{equation}\label{eq:S_ndef}
\mathcal{S}_n = \{x\in \N \mid A(n)(1 - \eps_n) \leq x \leq A(n) (1 + \eps_n) \}. 
\end{equation}
We note that
\begin{equation}\label{eq:S_nsize}
2\eps_nA(n) -1 \leq \card{\mathcal{S}_n} \leq  2\eps_nA(n) + 1.
\end{equation}

Now we explicitly give the probability mass function (PMF) for each $\mathbf{f}_n$. For any integer $m \in \Z$, we have
\begin{equation*}
    \mathbb{P}( \mathbf{f}_n = m) = \begin{cases}
        \frac{1 }{ \card{\mathcal{S}_n}} &\text{if } m \in \mathcal{S}_n, \\
        0 & \text{otherwise. }
    \end{cases}
\end{equation*}

We note that any possible sequence of outcomes from $\{\mathbf{f}_n\}_{n=1}^\infty$ will yield a function $f : \N \to \N$ such that $f(n) \sim A(n)$ as $n \to \infty$.  Conversely, any such function satisfying $f(n) \sim A(n)$ can arise as a possible sequence of outcomes from $\{\mathbf{f}_n\}_{n=1}^\infty$, for some suitable choice of the errors $\{\eps_n\}_{n=1}^\infty$.

We can now define the random variable
$$\mathbf{\Delta}_{\mathbf{f},k}(n) := \min \big\{  | \mathbf{f}_n- m^k| : m \in \Z  \big\} $$
analogously to $\Delta_{f,k}(n)$.

For a fixed $k \geq 2$, we compute the probability that $\mathbf{f}_n$ is within $d$ of a $k$th power.

\begin{lemma} \label{lem:prob_delta}
    Let $\{\mathbf{f}_n\}_{n=1}^\infty$ be a sequence of random variables as defined in \eqref{eq:S_ndef}.  Fix some integers $k \geq 2$ and $d \geq 0$.  Then there exists some constant $B > 0$ (independent of $k$ and $d$), such that for all $n > B$ we have the upper bound
\begin{equation*}
    \mathbb{P} ( \mathbf{\Delta}_{\mathbf{f},k}(n) \leq d ) \leq  \frac{2d+1}{k A(n)^{1 - \frac{1}{k}}} + \frac{6d+3}{2 \eps_n A(n)} + \frac{(8d+4)\eps_n^2}{A(n)^{1 - 1/k}}.
\end{equation*}
Furthermore, there exists a constant $B_d > 0$ (possibly depending on $d$), such that for all $n > B_d$, we have the lower bound
\begin{equation*}
    \mathbb{P} ( \mathbf{\Delta}_{\mathbf{f},k}(n) \leq d ) \geq \frac{2d+1}{k A(n)^{1 - \frac{1}{k}}} - \frac{10d+3}{2 \eps_n A(n)}.
\end{equation*}

\end{lemma}

\begin{proof}
    Fix some integers $k \geq 2$ and $d \geq 0$. We note that
\begin{equation*}
    \mathbb{P} ( \mathbf{\Delta}_{\mathbf{f},k}(n) \leq d ) 
    = \frac{\card{ \{ x \in \mathcal{S}_n : x \text{ is within $d$ of a $k$th power} \} } }{\card{ \mathcal{S}_n} }.
\end{equation*}
Recall from \eqref{eq:S_nsize} that for each $n$, 
\[
2\eps_nA(n) -1 \leq \card{\mathcal{S}_n} \leq  2\eps_nA(n) + 1.
\]
We now wish to count the number of $k$th powers in the set $\mathcal{S}_n$. For $k \geq 2$ set
\[
N_k(n):=\card {\{m\in \mathcal{S}_n : m = t^k\ \text{for some}\ t\in\mathbb{N}\} }.
\]
The total number of $k$th powers up to $x$ is $\lfloor x^{1/k}\rfloor$. Therefore, this gives the following lower and upper bounds for $N_k(n)$,
\begin{multline*}
(A(n)(1 + \varepsilon_n))^{1/k} -  (A(n)(1 - \varepsilon_n))^{1/k} -1 \leq N_k(n) \leq \\
(A(n)(1 + \varepsilon_n))^{1/k} -  (A(n)(1 - \varepsilon_n))^{1/k}+1.
\end{multline*}

Recall that, for any $x \in \R$ such that $|x| < 1$, we have the generalised binomial expansion
\begin{align*}
    (1 + x)^{1/k} = \sum_{i=0}^{\infty} \binom{1/k}{i} x^i 
\end{align*}
where $\binom{1/k}{i} = \frac{1/k (1/k - 1) \cdots (1/k - i + 1)}{i!}$ is the generalised binomial coefficient. Thus, if $|\eps_n| \leq 1/2$, we have that 
\begin{align*}
    (1 + \eps_n)^{1/k} - ( 1 - \eps_n)^{1/k} = \sum_{i=0}^{\infty} \binom{1/k}{i} \eps_n^i - \sum_{i=0}^{\infty} \binom{1/k}{i} (-\eps_n)^i = \frac{2 \eps_n}{k} + 2 \eps_n^3 \sum_{i=0}^{\infty} \binom{1/k}{2i + 3} \eps_n^{2i}.  
\end{align*}
Since the binomial coefficients satisfy $|\binom{1/k}{i}| \leq 1$ and given that $|\eps_n| \leq 1/2$, this gives us that
\begin{equation*}
    \sum_{i=0}^{\infty} \binom{1/k}{2i + 3} \eps_n^{2i} \leq \sum_{i=0}^{\infty}  \frac{1}{2^{2i}} < 2,
\end{equation*}
which proves that
\begin{equation*}
    \frac{2 \eps_n}{k} < (1 + \eps_n)^{1/k} - ( 1 - \eps_n)^{1/k} < \frac{2 \eps_n}{k} + 4 \eps_n^3 .
\end{equation*}

Thus, for all $n$ such that $|\eps_n| \leq 1/2$, we have the following lower and upper bounds for $N_k(n)$,
\begin{align*}
    \frac{2 \eps_n A(n)^{1/k}}{k} - 1 < 
    N_k(n) <  \frac{2 \eps_n A(n)^{1/k}}{k} +  4 \eps_n^3 A(n)^{1/k} + 1.
\end{align*}
We now note that the number of integers $x \in \mathcal{S}_n$ which are at most $d$ away from a $k$th power is at most $(2d+1) N_k(n)$.  Thus we have
\begin{equation*}
\card { \{ x \in \mathcal{S}_n : x \text{ is within $d$ of a $k$th power } \} } \leq
    (2d+1) N_k(n). 
\end{equation*}
This therefore gives the upper bound
\begin{align*}
    \mathbb{P}( \mathbf{\Delta}_{\mathbf{f},k}(n) \leq d  ) &=
    \frac{\card{ \{ x \in \mathcal{S}_n : x \text{ is within $d$ of a $k$th power} \} } }{\card{ \mathcal{S}_n} } 
    \leq \frac{(2d+1) N_k(n)}{2\eps_nA(n) - 1}
    \\
    &\leq  \frac{(2d+1) (2 \eps_n A(n)^{1/k}/k +  4 \eps_n^3 A(n)^{1/k} + 1 ) }{2 \eps_n A(n)} \cdot \frac{2 \eps_n A(n)}{2 \eps_n A(n) - 1} \\
    &\leq \left(  \frac{(2d+1)}{k A(n)^{1 - 1/k}} + \frac{(2d+1)}{2 \eps_n A(n)} + \frac{(4d+2) \eps_n^2}{A(n)^{1 - 1/k}} \right) \cdot \left(  1 + \frac{1}{2 \eps_n A(n) - 1} \right).
\end{align*}
As $\eps_n A(n) \gg \sqrt{A(n)}$, then for $n$ sufficiently large, we have that $2 \eps_n A(n) - 1 \geq \eps_n A(n)$.  This gives us the upper bound
\begin{align*}
    \mathbb{P}( \mathbf{\Delta}_{\mathbf{f},k}(n) \leq d ) &\leq \frac{(2d+1)}{k A(n)^{1 - 1/k}} + \frac{(2d+1)}{2 \eps_n A(n)} + \frac{(4d+2) \eps_n^2}{A(n)^{1 - 1/k}} \\
    &\quad +  \frac{(2d+1)}{k \eps_n A(n)^{2 - 1/k}} + \frac{(2d+1)}{2 \eps_n^2 A(n)^{2}} + \frac{(4d+2) \eps_n}{A(n)^{2 - 1/k}} .
\end{align*}
If we now choose $n$ sufficiently large such that $|A(n)| > 1$ and $|\eps_n A(n)| > 1$, then we can combine the fourth and fifth terms above into the second term, and the sixth term above into the third term.  This gives us the upper bound
\begin{align*}
    \mathbb{P}( \mathbf{\Delta}_{\mathbf{f},k}(n) \leq d ) \leq \frac{2d+1}{k A(n)^{1 - 1/k}} + \frac{6d+3}{2 \eps_n A(n)} + \frac{(8d+4) \eps_n^2}{A(n)^{1 - 1/k}}
\end{align*}
for all $n$ sufficiently larger than some absolute constant $B$, not depending on $d$ or $k$.

To obtain a lower bound for $\mathbb{P}( \mathbf{\Delta}_{\mathbf{f},k}(n) \leq d )$ we assume that $n$ is sufficiently large such that $\sqrt{A(n)} > 10(2d+1)$ and $\eps_n \leq 1/2$. In particular, this implies that we have
\begin{equation} \label{eq:lower_bound_prob_condition}
    \ceil{ (A(n) (1 - \eps_n ))^{1/k} }^k - \ceil{ (A(n) (1 - \eps_n ))^{1/k} - 1}^k \geq 2d+1.
\end{equation}
This implies that every $k$th power in $\mathcal{S}_n$ will be at least $2d$ away from another $k$th power. This therefore gives the lower bound
\begin{align*}
    (2d+1) N_k(n) - 2d \leq & \card { \{ x \in \mathcal{S}_n : x \text{ is within $d$ of a $k$th power } \} }. 
\end{align*}
Similarly, we thus obtain a lower bound for $\mathbb{P}( \mathbf{\Delta}_{\mathbf{f},k}(n) \leq d  )$,
\begin{align*}
    \mathbb{P}( \mathbf{\Delta}_{\mathbf{f},k}(n) \leq d  ) &=
    \frac{\card{ \{ x \in \mathcal{S}_n : x \text{ is within $d$ of a $k$th power} \} } }{\card{ \mathcal{S}_n} } 
    \geq \frac{(2d+1) N_k(n) - 2d}{2\eps_nA(n) + 1}
    \\
    &\geq  \frac{(2d+1) (2 \eps_n A(n)^{1/k}/k - 1 ) - 2d}{2 \eps_n A(n)} \cdot \frac{2 \eps_n A(n)}{2 \eps_n A(n) + 1} \\
    &\geq \left(  \frac{(2d+1)}{k A(n)^{1 - 1/k}} - \frac{(4d+1)}{2 \eps_n A(n)}  \right) \cdot \left(  1 - \frac{1}{2 \eps_n A(n) + 1} \right).
\end{align*}
Again, as $2 \eps_n A(n) + 1 \geq \eps_n A(n)$, we have that
\begin{align*}
    \mathbb{P}( \mathbf{\Delta}_{\mathbf{f},k}(n) \leq d  ) &\geq \frac{(2d+1)}{k A(n)^{1 - 1/k}} - \frac{(4d+1)}{2 \eps_n A(n)} - \frac{(2d+1)}{k \eps_n A(n)^{2 - 1/k}} - \frac{(4d+1)}{2 \eps_n^2 A(n)^2}.
\end{align*}
By similarly assuming that $|\eps_n A(n)| > 1$ and absorbing the second and fourth terms into the second, we obtain the lower bound
\begin{align*}
    \mathbb{P}( \mathbf{\Delta}_{\mathbf{f},k}(n) \leq d  ) &\geq \frac{2d+1}{k A(n)^{1 - 1/k}} - \frac{10d+3}{2 \eps_n A(n)}
\end{align*}
for all $n$ larger than $B$ and such that condition (\ref{eq:lower_bound_prob_condition}) holds. This proves the lemma.
\end{proof}

By thus considering the asymptotic limit as $n \to \infty$ in Lemma \ref{lem:prob_delta}, we obtain the following corollary.

\begin{cor} \label{cor:prob_delta}
    Let $\{\mathbf{f}_n\}_{n=1}^\infty$ be a sequence of random variables as defined in \eqref{eq:S_ndef}.  Fix some integers $k \geq 2$ and $d \geq 0$. Then
    \begin{equation*}
    \mathbb{P} ( \mathbf{\Delta}_{\mathbf{f},k}(n) \leq d ) 
    = \frac{2d+1}{k A(n)^{1 - \frac{1}{k}}}  + O_d \left( \frac{1}{\eps_n A(n)} + \frac{ \eps_n^2}{A(n)^{ 1 - \frac1k}}  \right)
\end{equation*}
as $n \to \infty$.  In particular, in the case $d = 0$, we get
\begin{equation*}
    \mathbb{P}(\mathbf{f}_n \text{is a $k$th power}) = \frac{1}{k A(n)^{1 - \frac{1}{k}}} + O \left( \frac{1}{\eps_n A(n)} + \frac{ \eps_n^2}{A(n)^{1 - \frac{1}{k}}}  \right)
\end{equation*}
as $n \to \infty$.
\end{cor}

We can also compute the probability that $\mathbf{f}_n$ is within $d$ of any perfect power.

\begin{lemma} \label{lem:prob_delta2}
    Let $\{\mathbf{f}_n\}_{n=1}^\infty$ be a sequence of random variables as defined in \eqref{eq:S_ndef}.  Fix some integer $d \geq 0$. Then there exists some constant $A_d > 0$ (possibly depending on $d$) such that, for all $n > A_d$, we have the upper bound
    \begin{multline} \label{eq:upper_delta_2}
        \mathbb{P} ( \mathbf{\Delta}_{\mathbf{f},k}(n) \leq d \text{ for some } k \geq 2) \\ \leq \frac{2d+1}{2 \sqrt{A(n)}} +  \frac{2d+1}{A(n)^{2/3}} + \frac{(8d+5) \eps_n^2}{\sqrt{A(n)}} + \frac{(6d+3) \log_2 A(n)}{2 \eps_n A(n)}.
    \end{multline}
    Furthermore, there exists a constant $B_d > 0$ (possibly depending on $d$), such that for all $n > B_d$, we have the lower bound
    \begin{equation*}
        \mathbb{P} ( \mathbf{\Delta}_{\mathbf{f},k}(n) \leq d \text{ for some } k \geq 2) \geq \frac{2d+1}{2 \sqrt{A(n)}} - \frac{10d+3}{2 \eps_n A(n)} .
    \end{equation*}
\end{lemma}

\begin{proof}
    To show an upper bound, we can apply the following union bound,
    \begin{align*}
        \mathbb{P} ( \mathbf{\Delta}_{\mathbf{f},k}(n) \leq d \text{ for some } k \geq 2) \leq \sum_{k=2}^{\infty}  \mathbb{P} ( \mathbf{\Delta}_{\mathbf{f},k}(n) \leq d).
    \end{align*}

    Now as $\mathbf{f}_n$ attains a value between $A(n) ( 1 - \eps_n)$ and $A(n)(1 + \eps_n)$, then if $A(n) ( 1 - \eps_n) > 1 + d$ and $A(n) ( 1 + \eps_n) < 2^k - d$, then $\mathbf{f}_n$ cannot be within $d$ to a $k$th power and thus $ \mathbb{P} ( \mathbf{\Delta}_{\mathbf{f},k}(n) \leq d) = 0$.  Therefore, assuming $n$ sufficiently large such that $|\eps_n| \leq 1/2$ and $A(n)/2 > 1 + d$, then we have
    \begin{multline*}
        \mathbb{P} ( \mathbf{\Delta}_{\mathbf{f},k}(n) \leq d \text{ for some } k \geq 2) \leq \sum_{k=2}^{\floor{\log_2 ( 2 A(n) + d )} }  \mathbb{P} ( \mathbf{\Delta}_{\mathbf{f},k}(n) \leq d) \\
        \leq  \sum_{k=2}^{\floor{\log_2 (A(n) + d )} + 1 } \left( \frac{2d+1}{k A(n)^{1 - \frac{1}{k}}} + \frac{6d+3}{2 \eps_n A(n)} + \frac{(8d+4)\eps_n^2}{A(n)^{1 - 1/k}} \right) .
    \end{multline*}
    By separating out the small terms for $k = 2$ and $k = 3$, we obtain the lower bound
    \begin{align*}
         \mathbb{P} ( \mathbf{\Delta}_{\mathbf{f},k}(n) \leq d &\text{ for some } k \geq 2) \leq \frac{2d+1}{2 \sqrt{A(n)}} +  \frac{2d+1}{3 A(n)^{2/3}} +  \frac{(2d+1) \log_2 (A(n) + d ) }{4 A(n)^{3/4}}  \\
         &\quad + \frac{(8d+4) \eps_n^2}{\sqrt{A(n)}} + \frac{(8d+4) \log_2 (A(n) + d ) \eps_n^2}{A(n)^{2/3}}  + \frac{(6d+3) \log_2 (A(n) + d )}{2 \eps_n A(n)} .
    \end{align*}
    By taking $n$ large enough so that $A(n)^{1/12} > \log_2 (A(n) + d )$, we can give an upper bound of
    \begin{align*}
         \mathbb{P} ( \mathbf{\Delta}_{\mathbf{f},k}(n) \leq d \text{ for some } k \geq 2) &\leq \frac{2d+1}{2 \sqrt{A(n)}} +  \frac{2d+1}{A(n)^{2/3}} + \frac{(8d+5) \eps_n^2}{\sqrt{A(n)}} + \frac{(6d+3) \log_2 A(n)}{2 \eps_n A(n)}.
    \end{align*}
    To show the lower bound, we can simply consider the case for $k = 2$ and apply the lower bound from Lemma~\ref{lem:prob_delta},
    \begin{align*}
        \mathbb{P} ( \mathbf{\Delta}_{\mathbf{f},k}(n) \leq d \text{ for some } k \geq 2) \geq \mathbb{P} ( \mathbf{\Delta}_{\mathbf{f},2}(n) \leq d) 
        &\geq \frac{2d+1}{2 \sqrt{A(n)}} - \frac{10d+3}{2 \eps_n A(n)}.
    \end{align*}
\end{proof}

\noindent \textbf{Remark.}  One can also similarly prove bounds for $\mathbb{P} ( \mathbf{\Delta}_{\mathbf{f},k}(n) \leq d \text{ for some } k \geq 2)$ by noting that
\begin{equation*}
        \mathbb{P}( \mathbf{\Delta}_{\mathbf{f},k}(n) \leq d \text{ for some } k \geq 2 ) = \frac{\card{ \{ x \in \mathcal{S}_n : x \text{ is within $d$ of a perfect power} \} } }{\card{ \mathcal{S}_n} }
\end{equation*}
and then obtaining estimate on the size of the numerator by using results of Nyblom \cite{Nyblom, Nyblom2008} and Jakimczuk  \cite{Jakimczuk} to count the number of perfect powers in the set $\mathcal{S}_n$.

By again considering the asymptotic limit as $n \to \infty$ in Lemma~\ref{lem:prob_delta2}, we also obtain the following corollary.

\begin{cor} \label{cor:prob_delta2}
    Let $\{\mathbf{f}_n\}_{n=1}^\infty$ be a sequence of random variables as defined in \eqref{eq:S_ndef}.  Fix some integer $d \geq 0$. Then
    \begin{equation*}
        \mathbb{P} ( \mathbf{\Delta}_{\mathbf{f},k}(n) \leq d \text{ for some } k \geq 2) = \frac{2d+1 }{2 \sqrt{A(n)}} + O_d \left( \frac{\eps_n^2}{\sqrt{A(n)}} + \frac{1}{A(n)^{2/3}} + \frac{\log A(n)}{\eps_n A(n)} \right)
    \end{equation*}
    as $n \to \infty$. In particular, in the case $d = 0$, we get
    \begin{equation*}
        \mathbb{P}( \mathbf{f}_n \text{ is a perfect power}) = \frac{1}{2\sqrt{A(n)}} + O \left( \frac{\eps_n^2}{\sqrt{A(n)}} + \frac{1}{A(n)^{2/3}} + \frac{\log A(n)}{\eps_n A(n)} \right).
    \end{equation*}
\end{cor}

\section{Heuristics}\label{sec:heuristics}

In this section we use the probabilistic model given in Section \ref{sec:model} to prove general heuristic results which give the parts of Theorem \ref{thm:introheuristic} as special cases.  As discussed in Remark \ref{rem:thm1.3}, these provide heuristic proofs for generalizations of  Conjectures~\ref{conj:MOT1}--\ref{conj:MOT3&4}, as well as heuristic support for Sun's conjecture (Conj.~\ref{conj:sun}), and possible generalizations. 

When $\{\mathbf{f}_n\}_{n=1}^\infty$ is a sequence of random variables as defined in \eqref{eq:S_ndef}, define the random variable
$$ \mathbf{M}_{\mathbf{f}, k}(d) := \max \{ n \in \N :  \mathbf{\Delta}_{\mathbf{f}, k}(n) \leq d \}, $$
allowing $\infty$ to be a possible outcome.  The following result proves part (a) of Theorem \ref{thm:introheuristic} by setting $\mathbf{f}_n=\mathbf{p}_n$.

\begin{lemma}\label{lem:Mkdprob}
Fix $d \geq 0$ and $k \geq 2$. For $\{\mathbf{f}_n\}_{n=1}^\infty$ defined as in \eqref{eq:S_ndef}, we have that $\mathbf{M}_{\mathbf{f}, k}(d)$ is finite with probability 1. 
\end{lemma}

\begin{proof}
Fix $ d \geq 0$ and $k \geq 2$. From \eqref{eq:eps_ncondition} we observe that $\epsilon_n A(n) \gg \sqrt{A(n)} \log A(n)$. Then using Corollary~\ref{cor:prob_delta}, we have as $n \to \infty$,
\begin{align*}
    \mathbb{P}( \mathbf{\Delta}_{\mathbf{f},k}(n) \leq d  )
    &\ll \frac{1}{\sqrt{A(n)}}.
\end{align*}
Therefore, from the assumption in \S \ref{ss:setup} that  $\sum A(n)^{-1/2}$ converges, we have that $\sum_{n=1}^{\infty} \mathbb{P}( \mathbf{\Delta}_{\mathbf{f},k}(n) \leq d  )$ converges.  By the Borel-Cantelli Lemma, $\mathbf{M}_{\mathbf{f},k}(d)$ is thus finite with probability 1.
\end{proof}

\noindent Letting $d = 0$ in Lemma \ref{lem:Mkdprob}, we obtain that $\mathbf{M}_{\mathbf{f},k}(0)$ is finite with probability 1.  This yields the following corollary which proves part (b) of Theorem \ref{thm:introheuristic} by setting $\mathbf{f}_n=\mathbf{p}_n$.
\begin{cor}
    If $k \geq 2$, then $\{\mathbf{f}_n\}_{n=1}^\infty$ as defined in \eqref{eq:S_ndef} will take only finitely many values that are perfect $k$th powers with probability 1.
\end{cor}

In order to consider all perfect powers instead of only $k$th powers, for $\{\mathbf{f}_n\}_{n=1}^\infty$ as defined in \eqref{eq:S_ndef}, we define the random variable
\begin{equation}\label{eq:RVMdef}
\mathbf{M}_{\mathbf{f}}(d) := \max \{ n :  \mathbf{\Delta}_{\mathbf{f},k}(n) \leq d \text{ for some } k \geq 2\},
\end{equation}
allowing $\infty$ to be a possible outcome.

\begin{lemma}\label{lem:Mdprob}
For $\{\mathbf{f}_n\}_{n=1}^\infty$ as defined in \eqref{eq:S_ndef},
we have that $\mathbf{M}_{\mathbf{f}}(d)$ is finite for all $d \geq 0$ with probability 1. 
\end{lemma}

\begin{proof}
Fix some integer $d \geq 0$. By Lemma~\ref{lem:prob_delta2}, we note that by using the lower bound on $\eps_n$, as $n \to \infty$,
\begin{equation*}
    \mathbb{P} ( \mathbf{\Delta}_{\mathbf{f},k}(n) \leq d \text{ for some } k \geq 2) 
    \ll \frac{1}{\sqrt{A(n)}} + \frac{\log A(n)}{\eps_n A(n)} \ll \frac{1}{\sqrt{A(n)}}.
\end{equation*}
As in the proof of Lemma \ref{lem:Mkdprob}, using that $A(n)\rightarrow \infty$, we have that $$\sum_{n=1}^{\infty} \mathbb{P}( \mathbf{\Delta}_{\mathbf{f},k}(n) \leq d \text{ for some } k \geq 2  )$$ converges. By the Borel-Cantelli Lemma, $\mathbf{M}_{\mathbf{f}}(d)$ is finite with probability 1.
\end{proof}

\noindent Setting $d = 0$ in Lemma \ref{lem:Mdprob}, we obtain the following corollary.

\begin{cor} \label{cor:finitely_many_powers}
    Any $\{\mathbf{f}_n\}_{n=1}^\infty$ as defined in \eqref{eq:S_ndef}  
    will take only finitely many values that are perfect powers with probability 1.
\end{cor}

We note that this motivates the expectation that many partition functions, while perhaps not directly satisfying an analogue of Sun's conjecture, as stated in Conjecture \ref{conj:sun}, should at at least have only finitely many exceptions. In \S \ref{ss:expectationforpp} we compute the expected number of perfect power values taken by several partition counting functions.

\subsection{Functions asymptotic to $an^{-b} \exp{(cn^\beta)}$}

We now restrict our attention to functions $f(n)$ which satisfy the asymptotic
\begin{equation}\label{eq:asymptotic_shape}
f(n) \sim an^{-b}\exp{(cn^\beta)}
\end{equation}
as $n \to \infty$, for some choice of positive real $a,b,c,\beta$. We note that many variations of partition functions $f(n)$ are known to satisfy (\ref{eq:asymptotic_shape}).

For example, let $\overline{p}(n)$ denote the number of overpartitions of $n$, i.e., the number of partitions of $n$ in which the first occurrence of each part may or may not be overlined.  Let $q(n)$ denote the number of partitions of $n$ into distinct parts, let $sc(n)$ denote the number of partitions into distinct odd parts (i.e., self-conjugate partitions), and let $r(n)$ denote the number of non-unitary partitions of $n$ (i.e. excluding the part 1).  Let $\text{pl}(n)$ denote the number of plane (planar) partitions of $n$.  Then Table~\ref{tab:asymptotics} shows the values of $a,b,c,\beta$ such that these functions are known to grow asymptotically as in \eqref{eq:asymptotic_shape}, where the constant $A_{\text{pl}}$ is given by (e.g. see \cite{MutafchievKamenov})

$$A_{\text{pl}} := \frac{\zeta(3)^{7/36} \exp{(\zeta'(-1))}}{ 2^{11/36} \sqrt{3 \pi}}. $$ 

\renewcommand{\arraystretch}{1.5}

\begin{table}[h!]
    \centering
    \begin{tabular}{|c|c|c|c|c|}
    \hline
    $f(n)$ & $a$ & $b$ & $c$ & $\beta$\\ \hline 

        $p(n)$ & $\frac{1}{4 \sqrt 3}$ & $1$ & $\frac{ \pi\sqrt 2}{\sqrt 3}$ & $\frac12$\\[1mm]  \hline

        $\overline{p}(n)$ & $\frac{1}{8}$ & $1$ & $\pi$ & $\frac12$\\[1mm] \hline

       $q(n)$ & $\frac{1}{4 \sqrt[4]3}$ & $\frac34$ & $\frac{\pi}{\sqrt 3}$ & $\frac12$\\[1mm] \hline

       $sc(n)$  & $\frac{1}{2^{\frac{7}{4}} 3^{\frac{1}{4}}}$  & $\frac34$ & $\frac{\pi}{\sqrt 6}$ & $\frac12$ \\[1mm] \hline

        $r(n)$ &   $\frac{\pi}{12 \sqrt{2}}$ & $\frac{3}{2}$ & $\frac{\pi \sqrt{2}}{\sqrt{3}}$ & $\frac{1}{2}$ \\[1mm] \hline

        $\text{pl}(n)$ & $A_{\text{pl}}$ & $\frac{25}{36}$ & $\frac{3 \zeta(3)^{1/3}}{2^{2/3}}$ & $\frac{2}{3}$ \\[.5mm] \hline
       
    \end{tabular}
    \caption{Examples of asymptotics for various partitions functions of the form \eqref{eq:asymptotic_shape}.}
    \label{tab:asymptotics}
\end{table}

To obtain bounds on various probabilities involving such functions, we first establish the following lemma.

\begin{lemma} \label{lem:int_bounds}
    Let $b, c, \beta$ be fixed real numbers with $c, \beta > 0$. Then for all sufficiently large $x$, we have
    \begin{equation*}
        \int_x^{\infty} t^b \exp(-c t ^\beta) \, dt 
        \leq 
        \frac{2}{\beta c}
        x^{b+1-\beta} \exp(-cx^\beta) .
    \end{equation*}
\end{lemma}

\begin{proof}
    We make the substitution $u = ct^\beta$.  Thus we get $t = (u/c)^{1/\beta}$, $du = c \beta t^{\beta-1} \, dt$, and $dt = \beta^{-1} u^{-1} (u/c)^{1/\beta} du$. Thus
    \begin{multline*}
    \int_x^{\infty} t^b \exp(-c t ^\beta) dt =  \int_{c x ^\beta}^{\infty} \beta^{-1} u^{-1} (u/c)^{(b+1)/ \beta} \exp{(-u)} du \\ = \beta^{-1} c^{-(b+1)/\beta} \Gamma \left( \frac{b+1}{\beta}, cx^\beta \right), 
    \end{multline*}
    where $\Gamma(s, z)$ denotes the upper incomplete gamma function $\int_z^\infty t^{s-1} e^{-t} dt$. Now, given the well-known asymptotic $\Gamma(s,z) / (z^{s-1} e^{-z}) \to 1$ as $z \to \infty$ (e.g. see \cite{Temme}), we have that $\Gamma(s,z) \leq 2 z^{s-1} e^{-z}$ for all sufficiently large $z$.  This gives the result.
\end{proof}

We now prove our main theorem for this section which proves part (c) of Theorem~\ref{thm:introheuristic} by setting $\mathbf{f}_n=\mathbf{p}_n$.

\begin{theorem}\label{thm:mainasymptotic}
    Let $A(n) = an^{-b} \exp(cn^\beta)$ with $a,b,c,\beta>0$, set $\{\eps_n\}_{n=1}^\infty$ a sequence of positive real numbers such that $\eps_n \gg A(n)^{-1/k}$ as $n \to \infty$, let $\{\mathbf{f}_n\}_{n=1}^\infty$ be as defined in \eqref{eq:S_ndef}, and fix $k \geq 2$. Then as $d \to \infty$, 
    \begin{equation*}
        \mathbf{M}_{\mathbf{f},k}(d) < (2^{1/\beta} + \eps) \left(\frac{k}{c(k - 1)} \right)^{\frac{1}{\beta}} (\log d)^{\frac{1}{\beta}},
    \end{equation*}
    with probability 1.
\end{theorem}

\begin{proof}
    Fix some $\eps > 0$.  For each $d \geq 0$, let $E_d$ denote the event that 
    $$\mathbf{M}_{\mathbf{f},k}(d) > (2^{1/\beta} + \eps) \left(\frac{k}{c(k - 1)} \right)^{\frac{1}{\beta}} (\log d)^{\frac{1}{\beta}}.$$
    We shall show that $\sum_{d = 0}^{\infty} \mathbb{P} E_d$ is finite, and thus by the Borel-Cantelli lemma, only finitely many of the events $E_d$ occur, with probability 1.

    For brevity, we shall write $X_d := (2^{1/ \beta} + \eps) \Big(\frac{k}{c(k - 1)} \Big)^{\frac{1}{\beta}} (\log d)^{\frac{1}{\beta}}$. Note that
    \begin{align*}
        \mathbb{P} E_d = \mathbb{P} \left( \mathbf{M}_{\mathbf{f},k}(d) > X_d \right) 
        = \mathbb{P} \left(  \bigcup_{n=X_d  }^{\infty}  \mathbf{\Delta}_{\mathbf{f},k}(n) \leq d  \right) \leq \sum_{n = X_d}^{\infty} \mathbb{P} \left( \mathbf{\Delta}_{\mathbf{f},k}(n) \leq d  \right),
    \end{align*}
    where the last inequality follows from the union bound.  By Lemma~\ref{lem:prob_delta}, there is an absolute constant $C$ (depending only on the function $A(n)$ and $\eps_n$, and not on $d$ or $k$) such that for all $n > C$,
    \begin{equation*}
        \mathbb{P} \left( \mathbf{\Delta}_{\mathbf{f},k}(n) \leq d  \right) \leq \frac{2d+1}{k A(n)^{1 - \frac{1}{k}}} + \frac{6d+3}{2 \eps_n A(n)} + \frac{(8d+4)\eps_n^2}{A(n)^{1 - 1/k}}.
    \end{equation*}
    We now consider some sufficiently large $d$ satisfying $X_d > C$.  Then using our hypothesis that $\eps_n \gg A(n)^{-1/k}$, we obtain that
    \begin{align*}
        \mathbb{P} E_d \leq \sum_{n = X_d}^{\infty}  \mathbb{P} \left( \mathbf{\Delta}_{\mathbf{f},k}(n) \leq d  \right) &\leq \sum_{n = X_d}^{\infty} \left( \frac{2d+1}{k A(n)^{1 - 1/k}} + \frac{6d+3}{2 \eps_n A(n)} + \frac{(8d+4)\eps_n^2}{A(n)^{1 - 1/k}} \right) \\
        &\leq 10d \sum_{n = X_d}^{\infty} \frac{1}{A(n)^{1 - 1/k}} \qquad (\text{as } \eps_n \to 0) \\
        &\leq \frac{10d}{A(X_d)^{1 - 1/k}} + \int_{X_d}^{\infty} \frac{10d}{A(t)^{1 - 1/k}} dt,
    \end{align*}
where we've used the integral test to obtain an upper bound for $\sum_{n = X_d}^{\infty} \frac{1}{A(n)^{1 - 1/k}}$, noting that $A(t)$ is eventually non-decreasing as $t \to \infty$.
     
    To estimate the first term, we note that there exists some $0 < \eps^{\prime} < (2^{\frac{1}{\beta}} + \eps)^{\beta} - 2$ such that  
    $$A(X_d)^{1 - 1/k} \geq d^{2+\eps'}$$
    for all sufficiently large $d$.
    To estimate the second term, we use the integral bounds from Lemma \ref{lem:int_bounds}.  In particular, we note that
    \begin{align*}
        \int_{X_d}^{\infty} \frac{1}{A(t)^{1 - 1/k}} dt &= a^{1/k - 1} \int_{X_d}^{\infty} t^{b(k-1)/k} \exp{\Big(\frac{ - c(k-1)}{k} \cdot t^\beta \Big)} dt \\
        &\leq a^{1/k - 1} \cdot 
        \frac{2k}{\beta c (k-1)}
        \cdot X_d^{\frac{b(k-1)}{k} + 1 - \beta} \cdot \exp{\Big( \frac{ - c(k-1)}{k} \cdot X_d^\beta \Big)}.
    \end{align*}
    By the definition of $X_d$, we note that
    \begin{equation*}
        \exp{\Big( \frac{ - c(k-1)}{k} \cdot X_d^\beta \Big)} = \frac{1}{d^{(2^{1/\beta}+\eps)^\beta}}.
    \end{equation*}
    Since $(2^{1/\beta} + \eps)^{\beta} > 2$, there exists some $\delta = \delta_{\beta, \eps}$ such that $$\exp{\Big( \frac{ - c(k-1)}{k} \cdot X_d^\beta \Big)} = \frac{1}{d^{2+\delta}},$$ and therefore for all sufficiently large $d$, we have the bound
    \begin{align*}
        10d \int_{X_d}^{\infty}  \frac{1}{A(t)^{1 - 1/k}} dt \ll_k  (\log{d})^{\frac{b(k-1)+k}{k\beta} -1}  \cdot \frac{1}{d^{1+\delta}} \leq \frac{1}{d^{1 + \eps^{''}}}, 
    \end{align*}
    for some $\eps^{\prime \prime} > 0$.  This gives the bound
    \begin{equation*}
        \mathbb{P} E_d \ll_k \frac{1}{d^{1 + \eta}} \; \text{where} \; \eta = \min(\eps^{\prime}, \eps^{\prime\prime}),
    \end{equation*}
    for all sufficiently large $d$. Thus, the sum $\sum_{d = 0}^{\infty} \mathbb{P} E_d$ converges, and by the Borel-Cantelli lemma, only finitely many of the events $E_d$ occur, with probability 1.
\end{proof}

Combining the (unconditional) lower bound for $M_{f,k}(d)$ shown in Proposition~\ref{asy_Mk} with Theorem \ref{thm:mainasymptotic} gives the following corollary.

\begin{cor}
    Let $A(n) = an^{-b} \exp(cn^\beta)$ and let $\{\eps_n\}_{n=1}^\infty$ be a sequence of positive real numbers such that $\eps_n \gg A(n)^{-1/k}$ as $n \to \infty$.  Let $\{\mathbf{f}_n\}_{n=1}^\infty$ be a sequence of random variables as defined in \eqref{eq:S_ndef}. Fix some positive integer $k \geq 2$.  Then 
    \begin{equation*}
        \mathbf{M}_{\mathbf{f},k}(d) \asymp (\log d)^{1/\beta}
    \end{equation*}
    as $d \to \infty$, with probability 1.
\end{cor}

We next consider for fixed $d \geq 0$, the boundedness of the sequence $\mathbf{M}_{\mathbf{f},k}(d)$ as $k \to \infty$.  The following result proves part (d) of Theorem \ref{thm:introheuristic} by setting $\mathbf{f}_n=\mathbf{p}_n$.

\begin{theorem}
    Let $\{\mathbf{f}_n\}_{n=1}^\infty$ be as defined in \eqref{eq:S_ndef}, and fix $d \geq 0$.  Then $\{ \mathbf{M}_{\mathbf{f},k}(d) \}_{k=2}^{\infty}$ is bounded with probability 1.
\end{theorem}

\begin{proof}
    By Lemma~\ref{lem:prob_delta2}, we have that there exists some constant $A_d$ such that equation (\ref{eq:upper_delta_2}) holds.   
    For each $k \geq 2$, let $E_k$ denote the event that
    \begin{equation*}
        \mathbf{M}_{\mathbf{f},k}(d) > A_d .
    \end{equation*}
    As in the proof of Theorem \ref{thm:mainasymptotic}, we shall show that $\sum_{k=2}^{\infty} \mathbb{P} E_k$ is finite, and thus by the Borel-Cantelli lemma, only finitely many of the events $E_k$ occur, with probability 1.  This will imply that $\mathbf{M}_{\mathbf{f}, k}(d)$ is eventually bounded for all sufficiently large $k$ with probability 1.

Using the union bound, we have that
    \begin{align*}
        \sum_{k=2}^{\infty} \mathbb{P} E_k = \sum_{k=2}^{\infty} \mathbb{P} ( \mathbf{M}_{\mathbf{f},k}(d) > L_d  ) &\leq  \sum_{k=2}^{\infty} \sum_{n=A_d+1}^{\infty} \mathbb{P} (\mathbf{\Delta}_{\mathbf{f},k}(n) \leq d ) \\
        &= \sum_{n=A_d+1}^{\infty} \sum_{k=2}^{\infty} \mathbb{P} (\mathbf{\Delta}_{\mathbf{f},k}(n) \leq d ),
    \end{align*}
    where the last equality conditionally follows by Fubini's theorem, once we've established absolute convergence of the double sum.

    By an argument analogous to that in the proof of Lemma~\ref{lem:prob_delta2}, we have
    \begin{equation*}
        \sum_{k=2}^{\infty} \mathbb{P} (\mathbf{\Delta}_{\mathbf{f},k}(n) \leq d )
        \leq \frac{2d+1}{2 \sqrt{A(n)}} +  \frac{2d+1}{A(n)^{2/3}} + \frac{(8d+5) \eps_n^2}{\sqrt{A(n)}} + \frac{(6d+3) \log_2 A(n)}{2 \eps_n A(n)}
    \end{equation*}
    for all sufficiently large $n$.  In particular, this implies that
    \begin{equation*}
        \sum_{k=2}^{\infty} \mathbb{P} E_k \ll_d \sum_{n = A_d + 1}^{\infty} \frac{1}{\sqrt{A(n)}}.
    \end{equation*}
    As the infinite sum $\sum \frac{1}{\sqrt{A(n)}}$ converges by assumption, this proves the theorem.
    
\end{proof}

\subsection{Computing expectation}\label{ss:expectationforpp}

We illustrate Corollary~\ref{cor:finitely_many_powers} for various partition functions by explicitly computing the expected number of perfect powers.  Let $\mathbf{X}_{\mathbf{f}}$ be a random variable denoting the number of perfect powers the sequence $\{ \mathbf{f}_n \}_{n = 1}^{\infty}$ takes.  While we know that $\mathbf{X}_{\mathbf{f}}$ is finite, we can calculate $\mathbb{E} \mathbf{X}_{\mathbf{f}} $ for a choice of functions $A(n)$ and $\eps_n$, using
\begin{equation} \label{eqn:expectation}
    \mathbb{E} \mathbf{X}_{\mathbf{f}} = \sum_{n=1}^{\infty} \mathbb{P} (\mathbf{f}_n \text{ is a perfect power} ) = \sum_{n=1}^{\infty} \frac{  \card{\{ x \in \mathcal{S}_n : x \text{ is a perfect power}  \} } }{\card{ \mathcal{S}_n }}.
\end{equation}

For each partition function $f(n)$, we set $A(n) = an^{-b} \exp{( c n^\beta)}$ for fixed $a,b,c,\beta$.  As the value of $\mathbb{E} \mathbf{X}_\mathbf{f}$ depends also on a choice of errors $\eps_n$, we chose to set $\eps_n$ as the second order term in the asymptotic expansion of $f(n)/A(n)$, generalizing the Rademacher expansion for $p(n)$.  In all cases considered, this was of the form $\eps_n = B/n^\beta$ for some suitable constant $B$.  We observe from \eqref{eqn:expectation} that the expectation for the number of perfect powers obtained by the functions considered is very small.

For the usual partition function $p(n)$, we set $a = 1/(4\sqrt{3})$, $b = 1$, $c = \pi \sqrt{2/3}$, and $\beta = 1/2$.  We then set $\eps_n$ to be the second order term in the Rademacher expansion of $p(n)/A(n)$, in particular $\eps_n = (\sqrt{3/2} \, \pi + \pi/(24 \sqrt{6}))/\sqrt{n}$.  The computation yields
\begin{equation*}
    \mathbb{E} \mathbf{X}_{\mathbf{p}} \approx 2.526. 
\end{equation*}

For the strict partition function  $q(n)$, we  set $a = 1/(4 \sqrt[4]{3} n^{3/4})$, $b = 3/4$, $c = \pi/\sqrt{3}$, and $\beta = 1/2$.  We then set $\eps_n = (\pi/(48 \sqrt{3}) - 3 \sqrt{3} / (8 \pi))/\sqrt{n}$. The computation yields
\begin{equation*}
    \mathbb{E} \mathbf{X}_{\mathbf{q}} \approx 5.805. 
\end{equation*}

For the self-conjugate partition function, we  set $a = 1/(2 \cdot 24^{1/4})$, $b = 3/4$, $c = \pi/\sqrt{6}$, and $\beta = 1/2$.  We then set $\eps_n = (3 \sqrt{6}/(8 \pi) + \pi / (48 \sqrt{6}) ) / \sqrt{n}$. The computation yields
\begin{equation*}
    \mathbb{E} \mathbf{X}_{\mathbf{sc}} \approx 13.426. 
\end{equation*}

For the non-unitary partition function, we  set $a = \pi/(12 \sqrt{2})$, $b = 3/2$, $c = \pi\sqrt{2/3}$, and $\beta = 1/2$.  We then set $\eps_n = (3 \sqrt{3/2} / \pi + 13 \pi / (24 \sqrt{6})) / \sqrt{n}$. The computation yields
\begin{equation*}
    \mathbb{E} \mathbf{X}_{\mathbf{r}} \approx 4.093. 
\end{equation*}

For the plane partition function, we set $a = A_{\text{pl}}$, $b = 25/36$, $c = 3 \zeta(3)^{1/3}/2^{2/3}$, and $\beta = 2/3$.  We then set $\eps_n = (277/(864 (2 \zeta(3))^{1/3} ) - \zeta(3)^{2/3} / (1440 \cdot 2^{1/3}) ) / n^{2/3}$.  The computation yields
\begin{equation*}
    \mathbb{E} \mathbf{X}_{\mathbf{pl}} \approx 0.758. 
\end{equation*}

\section{Computations}\label{sec:computations}

Here we provide computational data which supports Conjecture~\ref{conj:MOT1} as well as related conjectures in Section \ref{sec:conjectures} pertaining to other partition counting functions.  We also provide some data on additional related investigations. Throughout this section, conjectured values for $M_{f,k}(d)$ are determined by calculating $\max \{ n \leq B : \Delta_{f,k}(n) \leq d \}$ for finite $B$. For these calculations, $B = 10^7$ was used when $f(n)=p(n)$, and $B=10^5$ was used in all other cases.

Figure \ref{fig:pn_mkd} gives support for Conjecture \ref{conj:Mkdasymp} by showing the conjectured asymptotic for $M_k(d)$ as a dotted line and values of $M_k(d)$ as a solid line for various $k$.  In Figure \ref{fig:pn_mkd_diff} we plot the difference between $M_{k}(d)$ and the lower bound given in Lemma~\ref{lem:Delta_bound}.

\begin{figure}[h!]
    \centering
    \begin{subfigure}[t]{\linewidth}
        \centering
        \includegraphics[width=\linewidth]{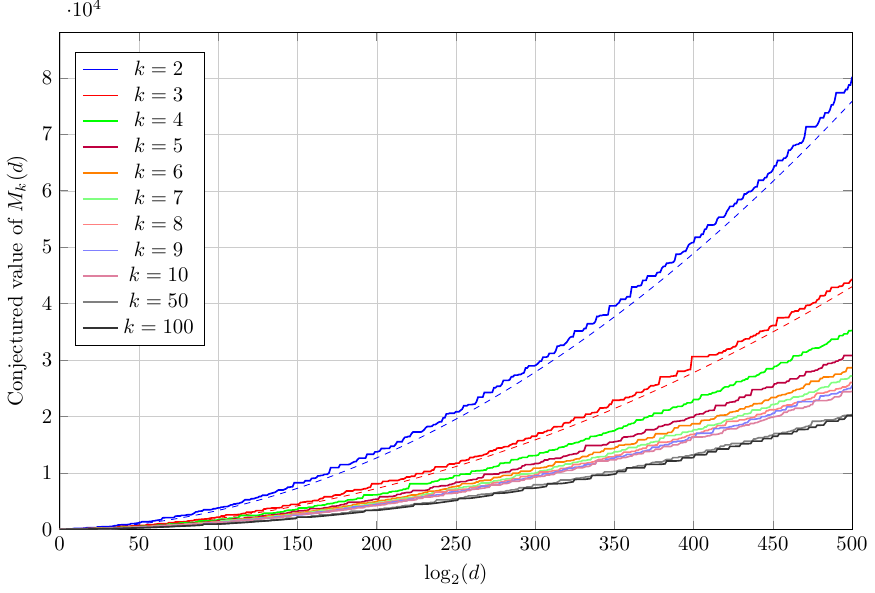}
        \caption{Conjectured value for $M_k(d)$ for the partition function $p(n)$.}
        \label{fig:pn_mkd}
    \end{subfigure}

    \vspace{0.8em}

    \begin{subfigure}[t]{\linewidth}
        \centering
        \includegraphics[width=\linewidth]{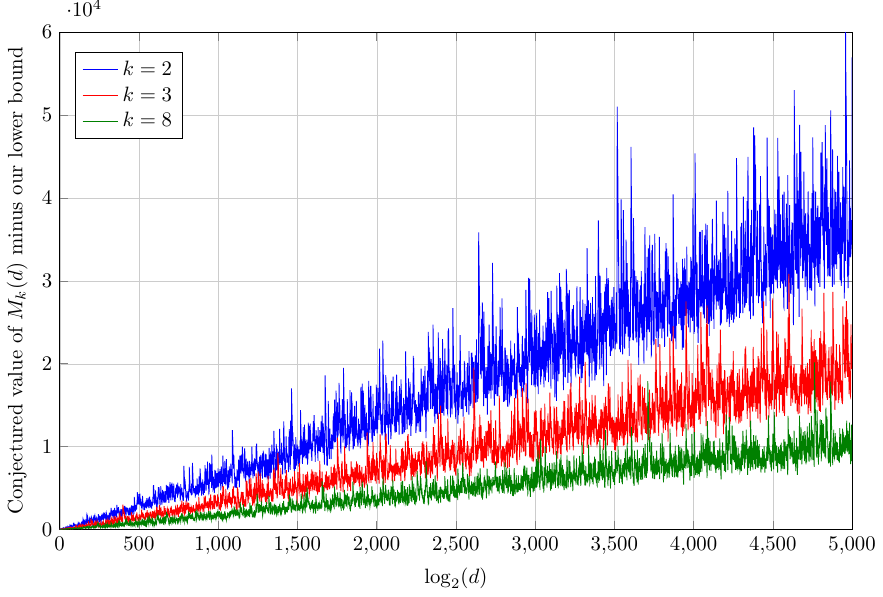}
        \caption{Conjectured value for $M_k(d)$ minus the lower bound $M^{\text{lower}}_{k}(d)$ for $p(n)$.}
        \label{fig:pn_mkd_diff}
    \end{subfigure}

    \caption{Computed values of $M_k(d)$ for the partition function $p(n)$, based on checking all $n \le 10^7$.}
    \label{fig:pn_mkd_combined}
\end{figure}

We further observe behavior similar to that in Figure \ref{fig:pn_mkd} for some other well-known partition counting functions.  We write $\overline{p}(n)$ to denote the number of overpartitions of $n$, $\text{pl}(n)$ for the number of plane partitions of $n$, $p_2(n)$ for the number of $2$-colored partitions of $n$, and $\text{sc}(n)$ for the number of self-conjugate partitions of $n$. 

Figures  \ref{fig:overpartitions}, \ref{fig:planepartitions}, and \ref{fig:twocoloredpartitions} show the conjectured $M_{f,k}(d)$ values for $d\leq 2^{450}$ for various $k$ values and for $f(n)=\overline{p}(n)$, $\text{pl}(n)$, $p_2(n)$, and $\text{sc}(n)$, respectively.

\begin{figure}[h!]
    \centering

    \begin{subfigure}{0.48\linewidth}
        \centering
        \includegraphics[width=\linewidth]{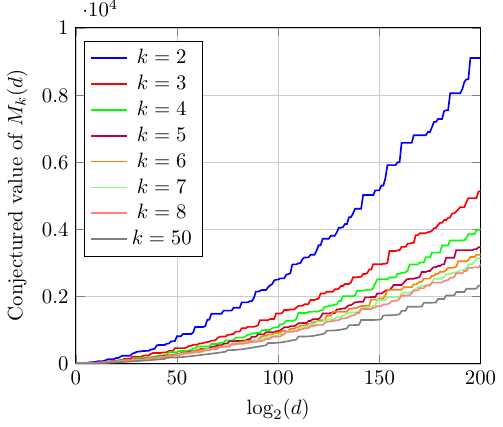}
        \caption{Conjectured value for $M_{\overline{p},k}(d)$ (overpartitions). OEIS: A015128.}
        \label{fig:overpartitions}
    \end{subfigure}
    \hfill
    \begin{subfigure}{0.48\linewidth}
        \centering
        \includegraphics[width=\linewidth]{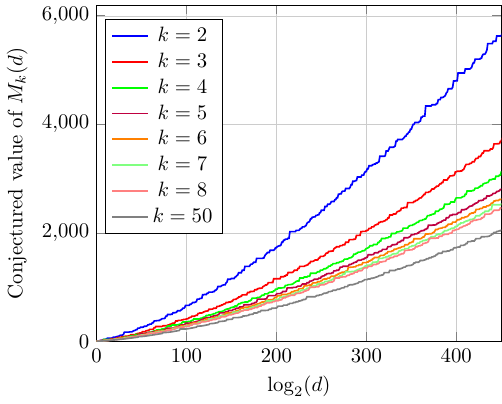}
        \caption{Conjectured value for $M_{\text{pl},k}(d)$ (plane partitions). OEIS: A000219.}
        \label{fig:planepartitions}
    \end{subfigure}

    \vspace{1em}

    \begin{subfigure}{0.48\linewidth}
        \centering
        \includegraphics[width=\linewidth]{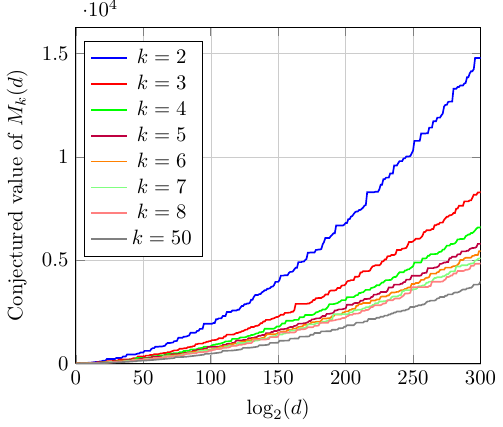}
        \caption{Conjectured value for $M_{p_2,k}(d)$ ($2$-colored partitions). OEIS: A000712.}
        \label{fig:twocoloredpartitions}
    \end{subfigure}
    \hfill
    \begin{subfigure}{0.48\linewidth}
        \centering
        \includegraphics[width=\linewidth]{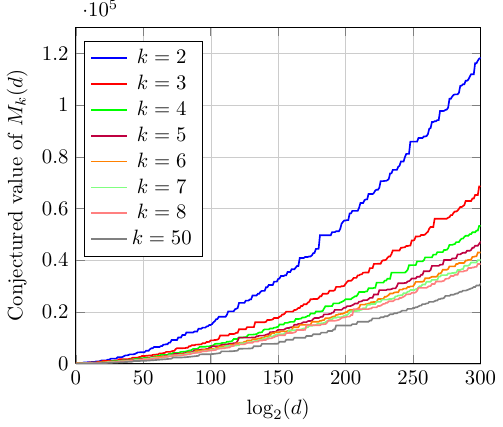}
        \caption{Conjectured value for $M_{\text{sc},k}(d)$ (self-conjugate partitions). OEIS: A000700.}
        \label{fig:selfconjpartitions}
    \end{subfigure}
    
    \caption{Conjectured values of $M_{f,k}(d)$ for some partition functions $f$, based on checking all $n \leq 10^5$.}
    \label{fig:partition_mkd}
\end{figure}

\section{Conjectures}\label{sec:conjectures}

In addition to Conjecture \ref{conj:Mkdasymp}, in this section we make a number of further conjectures motivated by the computational data discussed in Section \ref{sec:computations}, and in a similar spirit to the conjectures of Merca--Ono--Tsai \cite{MercaOnoTsai}. 

Our first conjecture is a direct analogue of Sun's conjecture, as stated in Conjecture \ref{conj:sun} for plane partitions.

\begin{conj}
The plane partition function $pl(n)$ is never a perfect power of the form $a^b$ for integers $a,b>1$.
\end{conj}

We also give an analogue of Conjecture \ref{conj:Mkdasymp} for the overpartition function $\overline{p}(n)$.

\begin{conj}
Let $k\geq 2$.  Then $M_{\overline{p},k}(d) \sim  \frac{1}{\pi^2}  \Big( \frac{k}{k-1} \Big)^2 (\log{d})^2$ as $d \to \infty$.
\end{conj}

We next consider analogues of Conjecture \ref{conj:MOT3&4} for overpartitions and $r$-colored partitions, i.e., that the hypothesis in Proposition \ref{prop:limit} that $\{ M_{f,k}(d)\}_{k=2}^\infty$ is satisfied for these counting functions. 

\begin{conj}
For each non-negative integer $d$, there is a positive integer $d$ such that for all $k \geq N_d$, we have 
\[
M_{\overline{p},k}(d) = L_{\overline{p}}(d).
\]
\end{conj}
Further, for any given non-negative integers $n$ and $d$, we have that $L_{\overline{p}}(d) = n$ is repeated $\overline{p}(n + 1) - \overline{p}(n)$ times.

\begin{conj}
For each non-negative integer $d$, there is a positive integer $d$ such that for all $k \geq N_d$, we have 
\[
M_{p_r,k}(d) = L_{p_r}(d).
\]
\end{conj}
Further, for any given non-negative integers $n$ and $d$, we have that $L_{p_r}(d) = n$ is repeated $p_r(n + 1) - p_r(n)$ times.

We further take a different approach by considering the following function analogous to $\Delta_{f,k}(n)$.  Let 
\[
\widetilde{\Delta}_{f,a}(n) := \min \{ |f(n) - a^k| : k \in \N \}.
\]
Here instead of fixing the power $k$ and letting the base vary as we find the distance between $p(n)$ and the closest $k$th power, we instead fix the base $a$ and let the power vary as we find the distance between $p(n)$ and the closest power of $a$.  Table \ref{tab:tilde_delta} shows values of $\widetilde{\Delta}_{a}(n)$ for $2\leq a \leq 9$ when $f(n)=p(n)$.

\renewcommand{\arraystretch}{1.3}

\small

\begin{table}[h]
\centering
\begin{tabular}{|c|c|c|c|c|c|c|c|c|c|}
\hline
$n$ & $p(n)$&$\widetilde{\Delta}_{2}(n)$ & $\widetilde{\Delta}_{3}(n)$ & $\widetilde{\Delta}_{4}(n)$& $\widetilde{\Delta}_{5}(n)$& $\widetilde{\Delta}_{6}(n)$& $\widetilde{\Delta}_{7}(n)$& $\widetilde{\Delta}_{8}(n)$& $\widetilde{\Delta}_{9}(n)$ \\ \hline
10 & 42& 10& 15& 22& 17& 6& 7& 22& 33 \\ \hline
20 & 627& 115& 102& 371& 2& 411& 284& 115& 102\\ \hline
30 & 5604& 1508& 957& 1508& 2479& 2172& 3203& 1508& 957\\ \hline
40 & 37338& 4570& 17655& 20954& 21713& 9318& 20531& 4570& 21711 \\ \hline
50 & 204226& 57918& 27079& 57918& 126101& 75710& 86577& 57918& 145177\\ \hline
\end{tabular}
\caption{Sample values for $\widetilde{\Delta}_{a}(n)$ }
\label{tab:tilde_delta}
\end{table}

\normalsize

We further define the related function analogous to $M_{f,k}(d)$,
\[
\widetilde{M}_{f,a}(d)  := \max \{n\in\N: \widetilde{\Delta}_a(n)\leq d\},
\]
allowing $\infty$ as a value. Table \ref{tab:tild_M_a} show conjectured values for $\widetilde{M}_{f,a}(d)$ for $2\leq a \leq 8$ and $d\leq 10^{200}$ when $f(n)=p(n)$.  These values are necessarily true for $n\leq 150\,000$.

\small

\begin{table}[h!]
\centering
\begin{tabular}{|c|c|c|c|c|c|c|c|c|c|c|c|}
\hline
$d$ & $\widetilde{M}_{2}(d)$ & $\widetilde{M}_{3}(d)$&$\widetilde{M}_{4}(d)$ & $\widetilde{M}_{5}(d)$ & $\widetilde{M}_{6}(d)$ &$\widetilde{M}_{7}(d)$ &$\widetilde{M}_{8}(d)$ &$\widetilde{M}_{9}(d)$ & $\widetilde{M}_{10}(d)$ \\ \hline

$10^0$ & 7 & 2 & 7& 2& 5& 2& 5& 2& 13\\ \hline

$10^5$ & 65 & 71 & 61& 65& 64& 59& 65& 71& 60\\ \hline

$10^{10}$ & 170 & 170 & 170& 164& 196& 180& 156& 170& 160\\ \hline

$10^{20}$ & 526 & 566 & 513& 513& 518& 491& 513& 566& 527\\ \hline

$10^{30}$ & 1085 & 1007 & 1031& 1005& 987& 1043& 1013& 1007& 980\\ \hline

$10^{40}$ & 1722 & 1710 & 1722& 1706& 1655& 1667& 1677& 1710& 1711\\ \hline

$10^{50}$ & 2557 & 2550 & 2529& 2720& 2562& 2509& 2502& 2507& 2550\\ \hline

$10^{60}$ & 3652 & 3555 & 3619& 3575& 3496& 3519& 3586& 3555& 3552\\ \hline

$10^{70}$ & 4749& 4723& 4749& 4893& 4764& 4697& 4749& 4665& 4718\\ \hline

$10^{80}$ & 6072& 6056& 6030& 6114& 6115& 6043& 6072& 6056& 6047\\ \hline

$10^{90}$ & 7795& 7551& 7556& 7471& 7511& 7556& 7699& 7478& 7538\\ \hline

$10^{100}$ & 9358& 9210& 9358& 9445& 9182& 9237& 9358& 9210& 9192\\ \hline

$10^{110}$ & 11120& 11123& 11120& 10978& 11167& 11084& 11006& 11123& 11009\\ \hline

$10^{120}$ & 13345& 13017& 13157& 12911& 13023& 13099& 13345& 13017& 12988\\ \hline

$10^{130}$ & 15635& 15378& 15635& 15308& 15192& 15093& 15297& 15060& 15129\\ \hline

$10^{140}$ & 17453& 17476& 17453& 17573& 17528& 17427& 17382& 17476& 17433\\ \hline

$10^{150}$ & 20443& 20193& 20443& 19993& 19833& 19928& 19828& 20071& 20153\\ \hline

$10^{160}$ & 22840& 22585& 22840& 22569& 22697& 22595& 22676& 22585& 22526\\ \hline

$10^{170}$ & 25369& 25521& 25369& 25300& 25307& 25430& 25199& 25521& 25315\\ \hline

$10^{180}$ & 28670& 28490& 28395& 28397& 28293& 28688& 28670& 28345& 28267\\ \hline

$10^{190}$ & 31687& 31775& 31398& 31449& 31445& 31598& 31494& 31622& 31380\\ \hline

$10^{200}$ & 35057& 34916& 34753& 34656& 35024& 35217& 35057& 34756& 34656\\ \hline

\end{tabular}
\caption{Conjectured $\widetilde{M}_{a}(d)$ values}
\label{tab:tild_M_a}
\end{table}

\normalsize 

\newcommand{\etalchar}[1]{$^{#1}$}

\end{document}